\documentclass[11pt, a4paper]{article}

\usepackage{amsmath, amsthm, amsfonts, graphicx, fullpage, hyperref, todonotes}
\usepackage[inline]{enumitem}

\numberwithin{equation}{section}

\newcommand{\bigO}{\mathcal{O}}
\newcommand{\gfn}{\mathbf{g}}
\newcommand{\be}{\begin{equation}}
\newcommand{\ee}{\end{equation}}

\newcommand{\realR}{\mathbb{R}}
\newcommand{\compC}{\mathbb{C}}
\newcommand{\stripS}{\mathbb{S}}
\newcommand{\intZ}{\mathbb{Z}}

    \def\Re{{\rm Re \,}}
    \def\Im{{\rm Im \,}}
    \def\Ai{{\rm Ai \,}}

    \def\bigO{{\cal O}}

    \def\P2n{{\rm P}_{{\rm II}}^{(n)}}

    \newtheorem{theorem}{Theorem}[section]
    \newtheorem{lemma}[theorem]{Lemma}
    \newtheorem{lem}[theorem]{Lemma}
    \newtheorem{cor}[theorem]{Corollary}

    \newtheorem{Definition}[theorem]{Definition}
    
    \newtheorem{Remark}[theorem]{Remark}
    
    \newtheorem{Example}[theorem]{Example}
    
    \newtheorem{Assumptions}[theorem]{Assumptions}

    {\rm \trivlist \item[\hskip \labelsep{\bf Proof}]}%
    {\hspace*{\fill}$\Box$\endtrivlist}

    \DeclareMathOperator*{\Tr}{Tr}
    
\newcommand{\Jlike}{\mathbf{J}}

\newcommand{\Jinv}{\mathbf{I}}

\begin{document}

\title{Universality for random matrices with equi-spaced external source: a case study of a biorthogonal ensemble}

\author{Tom Claeys\thanks{Universit\'{e} Catholique de Louvain, Chemin du cyclotron 2, B-1348 Louvain-La-Neuve, Belgium \newline
    email: \href{mailto:tom.claeys@uclouvain.be}{\protect\nolinkurl{tom.claeys@uclouvain.be}}} \ and
  Dong Wang\thanks{School of Mathematical Sciences, University of Chinese Academy of Sciences, Beijing, P.~R.~China 100049 \newline
    email: \href{mailto:wangdong@wangd-math.xyz}{\protect\nolinkurl{wangdong@wangd-math.xyz}}}}

\maketitle

\begin{abstract}
  We prove the edge and bulk universality of random Hermitian matrices with equi-spaced external source. One feature of our method is that we use neither a Christoffel-Darboux type formula, nor a double-contour formula, which are standard methods to prove universality results for exactly solvable models. This matrix model is an example of a biorthogonal ensemble, which is a special kind of determinantal point process whose kernel generally does not have a Christoffel-Darboux type formula or double-contour representation. Our methods may showcase how to handle universality problems for biorthogonal ensembles in general.
\end{abstract}

\section{Introduction}

\paragraph{The model}

We consider the space of $n\times n$ Hermitian matrices with a probability measure of the form
\begin{equation}\label{prob}
\frac{1}{Z_n}\exp(-n\Tr \left[V(M)-AM\right])dM,
\end{equation}
where
\begin{equation}
  d M=\prod_{i<j}d \Re M_{ij}d \Im M_{ij}\ \prod_{j=1}^n d M_{jj}.
\end{equation}
This is a random matrix model with external source $A$, a deterministic Hermitian matrix which can be assumed diagonal without loss of generality. Throughout the paper, we assume the technical condition that the external field $V(x)$ is real analytic and strongly convex, which means that 
\begin{equation} \label{convex}
  V''(x) \geq c > 0
\end{equation}
for some constant $c$. This condition implies that $V$ grows sufficiently fast as $x\to\pm\infty$ for the probability measure \eqref{prob} to be well-defined. In this paper, we assume $A$ is given by
\begin{equation} \label{eq:A_is_diagonal}
  A = {\rm diag}(a_1, a_2, \dotsc, a_n),
\end{equation}
with the eigenvalues $a_1,\ldots, a_n$ of $A$ equi-spaced. Without loss of generality, we may then assume that
\begin{equation} \label{eq:a_j_are_equispaced}
  a_{j} = \frac{j-1}{n}, \quad \text{ $j = 1, 2, \dotsc, n$.}
\end{equation}

Random matrix models with external source were introduced in \cite{Brezin-Hikami98, Zinn_Justin97}. In the case of a Gaussian external field $V(x)=x^2/2$, a random matrix $M$ from the ensemble \eqref{prob} follows the same probability law as a Gaussian Unitary Ensemble (GUE) matrix summed with the deterministic matrix $A$. This model and its dynamical generalization called Dyson's Brownian motion, with deterministic and random external source $A$, have been important in the study of universality for Wigner matrices, see e.g.\ \cite{Erdos-Ramirez-Schlein-Tao-Vu-Yau10, Erdos-Ramirez-Schlein-Yau10, Erdos-Yau17, Johansson01a, Johansson12}.

Yet another interpretation in this case is that the eigenvalues of $M$ follow the same joint probability distribution as $n$ non-intersecting Brownian motions with the eigenvalues of $A$ as starting points and $0$ as the common endpoint \cite{Gautie-Le_Doussal-Majumdar-Schehr19, Grela-Majumdar-Schehr21}.

For general $V$, although the macroscopic large $n$ behavior of the eigenvalues has been studied in some generality \cite{Eynard-Orantin09}, their microscopic behavior was only described rigorously if $A$ has two different eigenvalues with equal multiplicity \cite{Bleher-Kuijlaars04, Bleher-Kuijlaars05,Aptekarev-Bleher-Kuijlaars05, Bleher-Kuijlaars07, Bleher-Delvaux-Kuijlaars10, Aptekarev-Lysov-Tulyakov11} and if $A$ has a sufficiently small number of non-zero eigenvalues (aka spiked model) \cite{Baik-Wang10a, Baik-Wang10, Bertola-Buckingham-Lee-Pierce11a, Bertola-Buckingham-Lee-Pierce11}.

\paragraph{Biorthogonal ensembles}

In this paper, we use the term biorthogonal ensemble for an $n$-particle probability distribution defined on an interval $I\subset\mathbb R$, which may be semi-infinite or $(-\infty, \infty)$, such that the joint probability density  is of the form
\begin{equation} \label{eq:biorthogonal}
  \frac{1}{Z_n} \prod_{i < j} (\lambda_j - \lambda_i)(f(\lambda_j) - f(\lambda_i)) \prod^n_{i = 1} w(\lambda_i) d\lambda_i,
\end{equation}
where $\lambda_1 \leq \lambda_2 \leq \dotsb \leq \lambda_n \,\in I$ are the positions of the particles, $w$ is a non-negative integrable weight function on $I$, and $f(\lambda)$ describes one part of the two-particle interaction in the form of $\lvert f(\lambda_j) - f(\lambda_i) \rvert$, while we note that the other part of the two-particle interaction is given by $\lvert \lambda_j - \lambda_i \rvert$. To make the probability density well defined, we may require that the function $f(x)$ is increasing as $x \in I$. This is a generalization of the $n$-particle distribution with the joint probability density
\begin{equation} \label{eq:orthogonal_poly}
  \frac{1}{Z_n} \prod_{i < j} (\lambda_j - \lambda_i)^2 \prod^n_{i = 1} w(\lambda_i) d\lambda_i,
\end{equation}
often referred to as an orthogonal polynomial ensemble, which corresponds to \eqref{eq:biorthogonal} with $f(x) = x$.

The biorthogonal ensembles defined by \eqref{eq:biorthogonal} were systematically studied by Borodin in \cite{Borodin99}, in particular for  $f(x) = x^{\theta}$ and $w(x)$ one of the classical Hermite, Laguerre and Jacobi weights. Before \cite{Borodin99}, the special case of \eqref{eq:biorthogonal} with $f(x) = x^{\theta}$ and $w(x) = x^{\alpha} e^{-x}$ was introduced by Muttalib in \cite{Muttalib95} from a physical point of view. Hence, the special case of \eqref{eq:biorthogonal} with $f(x) = x^{\theta}$ is also called the Muttalib-Borodin ensemble, see \cite{Forrester-Wang15, Kuijlaars-Molag19, Molag20, Claeys-Girotti-Stivigny19, Charlier-Lenells-Mauersberger19} for example. Biorthogonal ensembles and variations thereof are also considered in \cite{Beenakker97, Lueck-Sommers-Zirnbauer06, Bloom-Levenberg-Totik-Wielonsky17, Butez17, Credner-Eichelsbacher15, Eichelsbacher-Sommerauer-Stolz11, Betea-Occelli20, Betea-Occelli20a, Gautie-Le_Doussal-Majumdar-Schehr19, Grela-Majumdar-Schehr21}.

For the biorthogonal ensemble given by \eqref{eq:biorthogonal}, we define the monic biorthogonal polynomials $p_i$ and $q_i$ with $i = 0, 1, 2, \dotsc$ being the degree of $p_i$ and $q_i$, such that
\begin{equation} \label{eq:biorth_poly}
  \int^b_a p_i(x) q_j(f(x)) w(x) dx = h_i \delta_{ij}.
\end{equation}
We note that by the positivity of the distribution \eqref{eq:biorthogonal}, $p_i$ and $q_i$ are all uniquely defined. Thus, there is a well defined sum
\begin{equation} \label{eq:pre_kernel}
  K_n(x, y) = \sum^{n - 1}_{i = 0} \frac{1}{h_i} p_i(x) q_i(f(y)) \sqrt{w(x) w(y)}.
\end{equation}
Then by the general theory of determinantal point processes (see \cite{Tracy-Widom98} for example), the $n$-particle distribution \eqref{eq:biorthogonal} is a determinantal point process, whose correlation kernel can be taken to be $g(x) K_n(x, y) g(y)^{-1}$  for any nonzero function $g(x)$ on $I$.

We remark that the term ``biorthogonal ensemble'' was, to the best of the authors' knowledge, first introduced by Borodin in \cite{Borodin99}, and we use the term in a slightly more general sense than \cite{Borodin99}. Other authors have also used the term in more general contexts, like \cite{Desrosiers-Forrester08}. We stick to the definition given in \eqref{eq:biorthogonal}, so that the biorthogonal polynomials $p_i, q_i$ can be defined by \eqref{eq:biorth_poly}.

\paragraph{The model as a biorthogonal ensemble}

For general eigenvalues $a_1,\ldots, a_n$, the Harish-Chandra-Itzykson-Zuber formula can be used to express the joint probability distribution for the eigenvalues of the random matrices $M$ with distribution \eqref{prob} in the form \cite{Brezin-Hikami98, Harish-Chandra57, Itzykson-Zuber80}
\begin{equation}\label{jpdf0}
 \frac{1}{Z'_n} \det(e^{na_i\lambda_j})^n_{i,j=1} \, \prod_{i < j} (\lambda_j - \lambda_i) \, \prod_{j=1}^n e^{-nV(\lambda_j)}\ \prod_{j=1}^n d\lambda_j,
\end{equation}
where $Z'_n ={\rm const} \cdot Z_n \cdot \prod_{i<j} (a_j - a_i)$.
For our special choice of $A$, this becomes
\begin{equation}\label{jpdf}
  \frac{1}{Z'_n}\prod_{i<j} (\lambda_j - \lambda_i) \
  \prod_{i<j} (e^{\lambda_j}-e^{\lambda_i})\ \prod_{j=1}^n
  e^{-nV(\lambda_j)}\ \prod_{j=1}^n d\lambda_j,
\end{equation}
which is the biorthogonal enxemble \eqref{eq:biorthogonal} with $f(\lambda)=e^\lambda$ and varying weight function
\begin{equation} \label{eq:varying_weight}
  w(\lambda) = e^{-nV(\lambda)}.
\end{equation}

Let $p^{(n)}_i$ and $q^{(n)}_i$ be monic biorthogonal polynomials for the model as defined in \eqref{eq:biorth_poly}. Then, we have the specialization of \eqref{eq:pre_kernel}
\begin{equation} \label{kernel}
  K^{(n)}_n(x, y) = \sum^{n-1}_{j=0} \frac{1}{h_j^{(n)}} p_j^{(n)}(x)q_j^{(n)}(e^y) e^{-\frac{n}{2}(V(x) + V(y))}, \quad h^{(n)}_j = \int_{\mathbb R} p^{(n)}_j(x) q_j^{(n)}(e^x)e^{-nV(x)}dx,
\end{equation}
and we recall that the correlation kernel of the model can be expressed as $g(x) K^{(n)}_n(x, y) g(y)^{-1}$ for any nonzero function $g(x)$.

\paragraph{Goal of the paper}

Recently, variations of the Deift-Zhou steepest-descent method applied to vector-valued Riemann-Hilbert problems have been successfully applied to the asymptotic analysis of biorthogonal polynomials $p_j, q_j$ for some specific biorthogonal ensembles, like
\begin{enumerate}[label=(\roman*)]
\item 
  $f(x) = e^{ax}$, $I = (-\infty, +\infty)$, and $w(x) = e^{-nV(x)}$ with $V$ in a rather general class \cite{Claeys-Wang11},
\item
  $f(x) = x^{\theta}$ with $\theta \in (0, \infty)$, $I\subset(0,+\infty)$ compact, and $w(x)$ has $n$-independent Fisher-Hartwig singularities \cite{Charlier21}, and
\item
  $f(x) = x^{\theta}$ with integer $\theta$, $I = (0, +\infty)$, and $w(x) = x^{\alpha} e^{-nV(x)}$ with $V$ in a rather general class \cite{Wang-Zhang21}.
\end{enumerate}
These developments support our belief that the method will allow to asymptotically solve more general types of biorthogonal ensembles in the future.

For orthogonal polynomial ensembles defined by \eqref{eq:orthogonal_poly}, the asymptotics of the orthogonal polynomials almost immediately yield the asymptotics of the correlation kernel $K_n(x, y) = \sum^{n - 1}_{i = 0} \frac{1}{h_i} p_i(x) p_i(f(y)) \sqrt{w(x)w(y)}$, and then the ``universality'' of the ensemble is derived. This is essentially due to the celebrated Christoffel-Darboux formula:
\begin{equation}\label{eq:CD}
  \sum^{n - 1}_{i = 0} \frac{1}{h_i} p_i(x) p_i(y) = \frac{p_n(x) p_{n - 1}(y) - p_{n - 1}(x) p_n(y)}{h_n}.
\end{equation}
For many other determinantal point processes related to variations of orthogonal polynomials, like multiple orthogonal polynomials, the computation of the limit of the correlation kernel, which is equivalent to solving the universality of the process, is also achieved through an analogue of the Christoffel-Darboux type formula. However, a biorthogonal ensemble does not have a Christoffel-Darboux-like formula in general, so the asymptotics of biorthogonal polynomials do not directly entail the universality of the biorthogonal ensemble. For some determinantal point processes associated to classical (multiple) orthogonal polynomials, the correlation kernel has a double-contour integral formula, like our matrix model with quadratic $V(x)$ and the Muttalib-Borodin ensemble with linear $V(x)$. Of course, most biorthogonal ensembles, including our matrix model with non-quadratic $V(x)$, do not have a double-contour integral kernel formula.

Among all the ``universal'' correlation kernels occuring in $1$-dimensional determinantal point processes, the sine kernel in the bulk and the Airy kernel at the (soft) edge are the simplest ones, and also the most common ones. It is natural to expect that they also occur in the model considered in our paper.

The first successful computation of the limit of the correlation kernel without Christoffel-Darboux type formula or double-contour integral formula in a biorthogonal ensemble was done in \cite{Wang-Zhang21}. In \cite{Wang-Zhang21}, only the limiting kernel at the ``hard'' edge of the interval is computed, and it turns out that the hard edge case is technically simpler than the soft edge and bulk cases.
Another approach to derive bulk universality in orthoganal polynomial ensembles without relying on the Christoffel-Darboux formula was developed recently in \cite{Eichinger-Lukic-Simanek21,Swiderski-Trojan21, Swiderski-Trojan21a, Swiderski-Trojan22, Swiderski-Trojan22a}, and universality results in orthogonal polynomial ensembles using a perturbative approach were obtained in \cite{Lubinsky09}.

The main technical result of this paper is to prove the universality of the model introduced in \eqref{prob}--\eqref{eq:a_j_are_equispaced}, that is, the sine correlation kernel occurs in the bulk and the Airy correlation kernel occurs at the edges. All of our arguments are based on the asymptotics of $p^{(n)}_j$ and $q^{(n)}_j$ obtained in \cite{Claeys-Wang11}. Besides confirming the expected universality results for the model, our purpose is more on the methodology: we showcase two methods through which the asymptotics of the correlation kernel, or equivalently, the universality of the model, can be derived from the asymptotics of the biorthogonal polynomials. We hope to convince the reader that the asymptotics of biorthogonal polynomials derived by the Deift-Zhou steepest-descent method for vector-valued Riemann-Hilbert problems are essentially enough for the universality of the biorthogonal ensembles, with the help of arguments analogous to those in our paper. We also expect that arguments similar to ours can be applied to more complicated universal limiting kernels for other biorthogonal ensembles.

\subsubsection*{Statement of results}

Let $\mu$ be the unique equilibrium measure which minimizes
\begin{equation} \label{energy0}
  I(\mu) \equiv \frac{1}{2} \iint \log \frac{1}{\lvert y-x \rvert } d\mu(y)d\mu(x) + \frac{1}{2} \iint \log \frac{1}{\lvert e^y-e^x \rvert }  d\mu(y)d\mu(x) + \int V(x)d\mu(x)
\end{equation}
among all Borel probability measures $\mu$ supported on $\mathbb R$. If $V$ is strongly convex, it is shown in \cite{Claeys-Wang11} that $\mu$ is supported on a single interval $[a,b]$ and that it has a smooth density $\psi(x)$ there that vanishes like a square root function at the endpoints. To be precise,
\begin{equation} \label{psia}
  d\mu(x) = \psi(x)dx, \quad x \in (x, y) \quad \text{with} \quad
  \begin{aligned}
    \psi(x) \sim {}& \alpha(x-a)^{1/2}, & x \searrow & a, \\
    \psi(x) \sim {}& \beta(b-x)^{1/2}, & x \nearrow & b,
  \end{aligned}
\end{equation}
and $\alpha, \beta > 0$. The density $\psi$ and the constants $\alpha,\beta$ have explicit expressions which will be given in Section \ref{section: eq}. The measure $\mu$ describes the macroscopic behavior of the eigenvalues in \eqref{jpdf} as the dimension $n$ of the random matrices tends to infinity.

Our result is the following.
\begin{theorem}\label{theorem: univ}
  Let $V$ be strongly convex and real analytic on $\mathbb R$, and denote $[a,b]$ for the support of the equilibrium measure $\mu$.
  Fix $x^* \in (a, b)$. The correlation kernel $K_n$ defined in \eqref{kernel} has the following scaling limits:
  \begin{align}
    & \lim_{n\to\infty} \frac{e^\frac{F'(x^*)(\xi-\eta)}{\pi\psi(x^*)}}{\pi \psi(x^*)n}K^{(n)}_n\left(x^*+\frac{\xi}{\pi\psi(x^*)n}, x^*+\frac{\eta}{\pi\psi(x^*)n} \right) & = {}& \frac{\sin\pi(\xi-\eta)}{\pi(\xi-\eta)}, \label{limsin} \\
    & \lim_{n\to\infty} \frac{e^{n(F(u_n)-F(v_n))}}{(\pi \beta n)^{2/3}} K^{(n)}_n \left( b + \frac{\xi}{(\pi\beta n)^{\frac{2}{3}}}, b + \frac{\eta}{(\pi\beta n)^{\frac{2}{3}}} \right) & = {}& \frac{\Ai(\xi)\Ai'(\eta)-\Ai'(\xi)\Ai(\eta)}{\xi-\eta}, \label{limAiryb} \\
    & \lim_{n\to\infty} \frac{e^{n(F(\widehat u_n)-F(\widehat v_n))}}{(\pi \alpha n)^{2/3}}K^{(n)}_n \left( a - \frac{\xi}{(\pi\alpha n)^{\frac{2}{3}}}, a - \frac{\eta}{(\pi\alpha n)^{\frac{2}{3}}} \right) & = {}& \frac{\Ai(\xi)\Ai'(\eta)-\Ai'(\xi)\Ai(\eta)}{\xi-\eta}. \label{limAirya}
  \end{align}
  Here the density function $\psi$, the constants $\alpha$ and $\beta$ are defined by \eqref{psia}, the variables $u_n, v_n$ and $\widehat u_n, \widehat v_n$ depend on $\xi, \eta$ as
  \begin{equation}
    u_n=b+ \frac{\xi}{(\pi\beta n)^{2/3}},\quad  
    v_n=b+ \frac{\eta}{(\pi\beta n)^{2/3}},\quad  \widehat u_n=a- \frac{\xi}{(\pi\alpha n)^{2/3}}, \quad \widehat v_n=a- \frac{\eta}{(\pi\alpha n)^{2/3}},
  \end{equation}
  and the function $F$ is defined as
\begin{equation} \label{eq:defn_F}
  F(x) = \frac{1}{2} \int \log \left\lvert \frac{e^x - e^y}{x-y} \right\rvert d\mu(y).
\end{equation}
\end{theorem}

\subsubsection*{Outline and heuristics of the proofs}

In Section \ref{section: eq}, we will collect general results about the equilibrium measure $\mu$ and about the biorthogonal polynomials $p_j^{(n)}$ and $q_j^{(n)}$. Most of those results are taken from \cite{Claeys-Wang11}, or are consequences of results in that paper. In Section \ref{section: Airy}, we will prove the Airy kernel limit \eqref{limAiryb}, and in Section \ref{section: sine}, we will prove the sine kernel limit \eqref{limsin}. 

In the Airy case, the well-known identity
\begin{equation}\label{eq:AiryCD}
  \int_0^{\infty} \Ai(\xi + y)\Ai(\eta + y) dy = \frac{\Ai(\xi)\Ai'(\eta)-\Ai'(\xi)\Ai(\eta)}{\xi - \eta},
\end{equation}
which should be seen as a continuous analogue of the Christoffel-Darboux formula for orthogonal polynomials, is of crucial importance. Instead of proving convergence of the correlation kernel to the right hand side of \eqref{eq:AiryCD}, we directly prove that the sum in \eqref{kernel} converges to the integral on the left hand side of \eqref{eq:AiryCD}. This method is similar to the one from \cite{Wang-Zhang21} for the hard edge case.

For the sine kernel, we will write the correlation kernel as a double sum which can be thought of as an analogue of a Christoffel-Darboux formula. Although no exact cancellations take place in this double sum, it will turn out that only few terms contribute to its large $n$ limit, such that, at least asymptotically, it is reminiscent of the right hand side of \eqref{eq:CD}.

\section{Equilibrium measure and asymptotics for biorthogonal polynomials}\label{section: eq}

\subsection{Equilibrium measure}

In this section, we collect some properties about the equilibrium measure $\mu_t$ defined as the unique Borel probability measure on $\mathbb R$ minimizing 
\begin{equation} \label{energy}
  I(\mu) \equiv  \frac{1}{2} \iint \log \frac{1}{\lvert y-x \rvert } d\mu(y)d\mu(x) + \frac{1}{2} \iint \log \frac{1}{\lvert e^y-e^x \rvert }  d\mu(y)d\mu(x) +\frac{1}{t} \int V(x)d\mu(x).
\end{equation}
For $t=1$, $\mu_1 = \mu$, the equilibrium measure in \eqref{energy0} and \eqref{psia}. If $V$ is strongly convex, then so is $\frac{1}{t}V(x)$, and this means that we can use the results obtained in \cite{Claeys-Wang11} about the equilibrium measure for any $t>0$. Those results can be summarized as follows.

Set
\begin{equation}\label{sab}
  \Jlike_{c_1, c_0}(s) = c_1s + c_0 - \log\frac{s-\frac{1}{2}}{s+\frac{1}{2}},\quad \mathbb C\setminus\left[-\frac{1}{2},\frac{1}{2}\right],
\end{equation}
with $c_1 >0$, $c_0 \in \realR$, and the logarithm corresponding to arguments between $-\pi$ and $\pi$. 
For any $c_1 >0$, $c_0 \in \realR$,  $\Jlike_{c_1,c_0}(s)$ has two critical points $s^-(c_1)<s^+(c_1)$ on the real line, given by (cf.~\cite[(1.24)]{Claeys-Wang11}
\begin{equation} \label{eq:s_b_in_c_10}
s^\pm(c_1)=\pm\sqrt{\frac{1}{4}+\frac{1}{c_1}}.
\end{equation}
The inverse image of the interval $[\Jlike_{c_1,c_0}(s^-(c_1)), \Jlike_{c_1,c_0}(s^+(c_1))]$ under $\Jlike$ consists of two complex conjugate curves $\gamma_1=\gamma_1(c_1), \gamma_2=\gamma_2(c_1)$. We write $\gamma=\gamma(c_1)$ for the counterclockwise union of $\gamma_1$ and $\gamma_2$. 

For any $t > 0$, there are $c_0(t)\in\mathbb R$ and $c_1(t)>0$ uniquely determined  by (see \cite[Lemma 2]{Claeys-Wang11})
\begin{align}
  1= {}& \frac{c_1(t)}{2\pi i t }\int_{\gamma(c_1(t)} V'(\Jlike_{c_1(t),c_0(t)}(s))ds, \label{intro mod eq 1} \\
  1= {}& \frac{1}{2\pi i t}\int_{\gamma(c_1(t))} \frac{V'(\Jlike_{c_1(t),c_0(t)}(s))}{s-\frac{1}{2}}ds. \label{intro mod eq 2}
\end{align}
In the subsequent part of the paper, we write $\Jlike(s) := \Jlike_{c_1(t), c_0(t)}(s)$, $\gamma:=\gamma(c_1(t))$ when there is no possible confusion, hence suppressing the dependence on $t$ in our notations.

The equilibrium measure $\mu_t$ is supported on a single interval $[a(t),b(t)]$, and the endpoints are given by
\begin{equation} \label{def:ab}
  a(t)=\Jlike(s_a(t)),\qquad b(t)=\Jlike(s_b(t)),\qquad s_a(t):=s^-(c_1(t)),\qquad s_b(t):=s^+(c_1(t)),
\end{equation}
or equivalently in terms of $c_0(t), c_1(t)$ by (cf.~\cite[(1.21)--(1.22)]{Claeys-Wang11})
\begin{gather}
  c_0(t) = {} \frac{b(t) + a(t)}{2}, \label{eq:first_def_of_c_0} \\
 c_1(t) s_b(t) - \log \frac{s_b(t) - \frac{1}{2}}{s_b(t) + \frac{1}{2}}=\frac{b(t) - a(t)}{2}. \label{eq:first_def_of_c_1}
\end{gather} 
By taking derivatives on both sides of \eqref{eq:first_def_of_c_0} and \eqref{eq:first_def_of_c_1}, we obtain
\begin{equation} \label{eq:deri_a_b}
  a'(t)=c_0'(t) - c_1'(t)s_b(t), \quad b'(t)=c_0'(t) + c_1'(t)s_b(t).
\end{equation}  

Let us denote $\Jinv_+$ (resp.\ $\Jinv_-$) for the inverse of $\Jlike$ which maps $[a(t),b(t)]$ to $\gamma_1\subset\mathbb C^+$ (resp.\ $\gamma_2\subset\mathbb C^-$). The equilibrium measure $\mu_t$ is supported on $[a(t), b(t)]$ with a continous density $\psi_t$ that is given by (see \cite[Theorem 3]{Claeys-Wang11})
\begin{equation} \label{eq:density_psi_t}
  \psi_t(x) = \frac{1}{2\pi^2}\int^{b(t)}_{a(t)} V''(u)\log\left\lvert \frac{\Jinv_+(u) - \Jinv_-(x)}{\Jinv_+(u) - \Jinv_+(x)}\right\rvert du.
\end{equation}
It follows after a straightforward calculation that
\begin{align}
  \psi_t(x) = {}& \alpha(t)(x-a(t))^{1/2} (1 + \bigO(x - a(t)), & x \searrow & a(t), \label{psita} \\
  \psi_t(x) = {}& \beta(t)(b(t)-x)^{1/2} (1 + \bigO(x - b(t)), & x \nearrow & b(t), \label{psitb}
\end{align}
with $\alpha(t)$ and $\beta(t)$ given by 
\begin{align}
  \alpha(t) = {}& -\frac{1}{\pi^2c_1(t) \lvert s_a(t)\rvert^{1/2}} \int^{b(t)}_{a(t)} V''(y) \Im\frac{1}{\Jinv_+(y) - s_a(t)} dy > 0, \\
  \beta(t) = {}& -\frac{1}{\pi^2 c_1(t) \lvert s_b(t)^{1/2} \rvert} \int_{a(t)}^{b(t)} V''(y) \Im\frac{1}{\Jinv_+(y)-s_b(t)} dy > 0.
\end{align}
After the change of variables $y=\Jlike(s)$, noting that $s_a(t) = -s_b(t)$ and that 
\[\Jlike'(s)=c_1(t)(s\pm s_b(t))\left(\frac{\frac{1}{2}\mp s_b(t)}{s-\frac{1}{2}}+\frac{\frac{1}{2}\pm s_b(t)}{s+\frac{1}{2}}\right)
  ,\]
and introducing the integrals
\begin{equation}\label{PQ}
  P_t = \frac{1}{2\pi i} \int_\gamma \frac{V''(\Jlike_{c_1(t), c_0(t)}(s))}{s - \frac{1}{2}} ds, \quad Q_t = \frac{1}{2\pi i}\int_\gamma \frac{V''(\Jlike_{c_1(t), c_0(t)}(s))}{s + \frac{1}{2}} ds,
\end{equation}  
we can rewrite $\alpha(t), \beta(t)$ as 
\begin{equation} \label{alpha}
  \begin{split}
    \alpha(t) = {}& \frac{1}{\pi s_b(t)^{1/2}}\left( \left( \frac{1}{2}-s_b(t) \right) P_t + \left( \frac{1}{2}+s_b(t) \right) Q_t \right) \\
    = {}& \frac{\frac{1}{2}-s_b(t)}{\pi s_b(t)^{1/2}} \left( P_t - \left( 1 + \frac{c_1(t)}{2} + c_1(t) s_b(t) \right) Q_t \right), 
  \end{split}
\end{equation}
\begin{equation} \label{beta}
  \begin{split}
    \beta(t) = {}& \frac{1}{\pi s_b(t)^{1/2}}\left( \left( \frac{1}{2}+s_b(t) \right) P_t + \left( \frac{1}{2}-s_b(t) \right) Q_t \right) \\
    = {}& \frac{\frac{1}{2}+s_b(t)}{\pi s_b(t)^{1/2}} \left( P_t - \left( 1 + \frac{c_1(t)}{2} - c_1(t) s_b(t) \right) Q_t \right).
  \end{split}
\end{equation}

We first establish the monotonicity of the endpoints $a(t), b(t)$ as a function of $t$.
\begin{lemma} \label{lemma endpoints}
For any $t > 0$, the equilibrium measure $\mu_t$ is supported on the interval $[a(t), b(t)]$. Moreover, $a(t)$ and $b(t)$ depend smoothly on $t$ and
  \begin{equation}\label{a1b1der}
    a'(t) = \frac{\frac{1}{2} - s_b(t)}{\pi \alpha(t)s_b(t)^{1/2}}<0,\quad b'(t) = \frac{\frac{1}{2} + s_b(t)}{\pi \beta(t)s_b(t)^{1/2}}>0.
  \end{equation}
\end{lemma}
\begin{proof}
  The fact that the support of $\mu_t$ consists of a single interval for strongly convex $V$ was proved in \cite{Claeys-Wang11}.
  Let the mapping $F: (u, v) \to (U, V)$ be defined by 
  \begin{equation} \label{def F}
    (U, V) = F(u, v) = \left(\frac{u}{2\pi i }\int_\gamma V'(\Jlike_{u, v}(s))ds, \, \frac{1}{2\pi i }\int_\gamma \frac{V'(\Jlike_{u, v}(s))}{s-\frac{1}{2}}ds\right).
  \end{equation}
 The Jacobian matrix is given by
  \begin{equation} \label{Jacobian}
    \frac{\partial(U, V)}{\partial(u, v)} =
    \begin{bmatrix}
      \frac{1}{2\pi i }\int_\gamma [V'(\Jlike_{u, v}(s)) + V''(\Jlike_{u, v}(s)s] ds & \frac{1}{2\pi i }\int_\gamma u V''(\Jlike_{u, v}(s))ds \\
      \frac{1}{2\pi i }\int_\gamma V''(\Jlike_{u, v}(s) \frac{s}{s - \frac{1}{2}} ds & \frac{1}{2\pi i }\int_\gamma V''(\Jlike_{u, v}(s) \frac{1}{s - \frac{1}{2}} ds
    \end{bmatrix}.
  \end{equation}
  Noting that
  \begin{multline}
    \frac{1}{2\pi i }\int_\gamma \left[ V'(\Jlike_{u, v}(s)) + V''(\Jlike_{u, v}(s))s \right] ds - \frac{1}{2\pi i }\int_\gamma \frac{1}{2} V''(\Jlike_{u, v}(s)) \left( \frac{1}{s - \frac{1}{2}} + \frac{1}{s + \frac{1}{2}} \right) ds \\
    = \frac{1}{2\pi i }\int_\gamma \frac{d}{ds}[V'(\Jlike_{u, v}(s))s] ds = 0,
  \end{multline}
  \begin{multline}
    \frac{1}{2\pi i} \int_{\gamma} V''(\Jlike_{u, v}(s)) ds + \frac{1}{2\pi i }\int_\gamma \frac{1}{u} V''(\Jlike_{u, v}(s)) \left( \frac{1}{s + \frac{1}{2}} - \frac{1}{s + \frac{1}{2}} \right) ds \\
    = \frac{1}{2\pi i }\int_\gamma \frac{d}{ds}[u^{-1} V'(\Jlike_{u, v}(s))] ds = 0,
  \end{multline}
we have
  \begin{equation}
    \left. \frac{\partial(U, V)}{\partial(u, v)} \right\rvert_{u = c_1(t), \, v = c_0(t)} =
    \begin{bmatrix}
      \frac{1}{2}(P_t + Q_t) & P_t - Q_t \\
      \frac{1}{c_1(t)}(P_t - Q_t) + \frac{1}{2}P_t & P_t
    \end{bmatrix},
  \end{equation}
  and, with the help of expressions \eqref{alpha} and \eqref{beta} for $\alpha(t)$ and $\beta(t)$, we have
  \begin{equation}
    \left. \det \frac{\partial(U, V)}{\partial(u, v)} \right\rvert_{u = c_1(t), \, v = c_0(t)} = P_tQ_t - \frac{1}{c_1(t)}(P_t - Q_t)^2  = \pi s_b(t) \alpha(t) \beta(t) > 0.
  \end{equation}
  By \eqref{intro mod eq 1} and \eqref{intro mod eq 2}, $c_1(t)$ and $c_0(t)$ are defined by the relation $F(c_1(t),c_0(t))=(t,t)$, and they are smooth functions of $t$ as the Jacobian of $F$ is non-singular at $(c_1(t),c_0(t))$. Moreover, taking derivatives with respect to $t$ in $F(c_1(t),c_0(t))=(t,t)$, by \eqref{def F}, we have
  \begin{equation}
    c_1'(t) = \frac{Q_t}{P_tQ_t - \frac{1}{c_1(t)}(P_t - Q_t)^2}, \quad c_0'(t) = \frac{\frac{Q_t}{2} - \frac{1}{c_1(t)}(P_t - Q_t)}{P_tQ_t - \frac{1}{c_1(t)}(P_t - Q_t)^2}.
  \end{equation}
  At last, by \eqref{eq:deri_a_b}, we have
  \begin{align}
    a'(t) = {}& \frac{Q_t \left(\frac{1}{2} - s_b(t) + \frac{1}{c_1(t)} \right) - \frac{1}{c_1(t)}P_t}{P_tQ_t - \frac{1}{c_1(t)}(P_t - Q_t)^2} = \frac{1}{P_t - (1 + \frac{c_1(t)}{2} + c_1(t)s_b(t))Q_t}, \\
    b'(t) = {}& \frac{Q_t \left(\frac{1}{2} + s_b(t) + \frac{1}{c_1(t)} \right) - \frac{1}{c_1(t)}P_t}{P_tQ_t - \frac{1}{c_1(t)}(P_t - Q_t)^2} = \frac{1}{P_t - (1 + \frac{c_1(t)}{2} - c_1(t)s_b(t))Q_t}.
  \end{align}
  Recalling the expressions \eqref{alpha} and \eqref{beta} for $\alpha(t),\beta(t)$, we obtain \eqref{a1b1der}.
\end{proof}

The next lemma is about the position of the endpoints $a(t), b(t)$ of the support of the equilibrium measure $\mu_t$. Since we assume the external field $V$ is strongly convex, it has a unique absolute minimum that we denote as $x_{\min}$.  The external field $V(x)-tx$ has also a unique minimum, which we denote as $\widehat x_{\min}(t)$. Observe that $\widehat{x}_{\min}(t)$ is an increasing function of $t$, that $\widehat x_{\min}(t)>x_{\min}$ for $t>0$, and that $\lim_{t\to 0}\widehat x_{\min}(t)=x_{\min}$.
\begin{lemma}\label{lemma zeros}
  \begin{enumerate}[label=(\alph*)]
  \item \label{enu:lemma zeros_1}
    For any $t > 0$, $x_{\min}<b(t)$, $a(t)<\widehat x_{\min}(t)$. 
  \item \label{enu:lemma zeros_1a}
    $\lim_{t\searrow 0} a(t) = \lim_{t\searrow 0} b(t) = x_{\min}$.
  \end{enumerate}
\end{lemma}
\begin{proof}
\begin{itemize}
\item[\ref{enu:lemma zeros_1}]
  Assume first that $b(t) \leq x_{\min}$. Consider the logarithmic energy defined in \eqref{energy} and assume that $\mu_t$ is its minimizer supported on $[a(t), b(t)]$. We define the measure $\mu^s_t$ by its shifted version $d\mu^s_t(x) = d\mu_t(x-s)$ that is supported on $[a(t) + s, b(t) + s]$. If we substitute $\mu^s_t$ into $\mu$ in \eqref{energy}, then the energy $I(\mu^s_t)$ is a function of $s$. Since $\mu_t = \mu^0_t$ is the minimizer, it is necessary that $\left. \frac{d}{ds} I(\mu^s_t) \right\rvert_{s = 0} = 0$. However, we observe that the first term on the right side of \eqref{energy} is unchanged as $s$ varies; the second term on the right side of \eqref{energy} decreases as $s$ increases, and the last term there also decreases since $V'(x) < 0$ for $x \in [a(t), b(t))$. We thus have $\left. \frac{d}{ds} I(\mu^s_t) \right\rvert_{s = 0} < 0$, which is in contradiction with the fact that $\mu$ minimizes $I(\mu)$. It follows that $x_{\min}<b(t)$.
  
  If we replace $V(x)$ by $V(-x) + tx$ in \eqref{energy}, it is straightforward to verify that the minimizer $\mu_t$ with $d\mu_t(x) = \psi_t(x)dx$ supported on $[a(t), b(t)]$ is replaced by $\mu^r_t$ with $d\mu^r_t(x) = \psi_t(-x)dx$ supported on $[-b(t), -a(t)]$. It follows from the argument above that $-a(t)$ is bigger then the minimum of $V(-x)+x$, which is $-\widehat x_{\min}(t)$, and then $a(t) < \widehat{x}_{\min}(t)$.
  
\item[\ref{enu:lemma zeros_1a}]
  By \eqref{a1b1der}, we have that as $t\searrow 0$, $b(t) - a(t) \searrow l^0 \geq 0$, and then  by \eqref{eq:first_def_of_c_1} $c_1(t) \searrow c^0_1 \geq 0$. If $c^0_1 > 0$, then we can see that \eqref{intro mod eq 1} cannot hold if $t$ is small enough. Hence, as $t \searrow 0$, $a(t)$ and $b(t)$ converge to the same point. Combining this with the results in part \ref{enu:lemma zeros_1} and using the fact that $\lim_{t\to 0}\widehat x_{\min}(t)=x_{\min}$, we obtain the result.
\end{itemize}
\end{proof}

\subsection{Zeros and asymptotics of biorthogonal polynomials}

Detailed uniform in $z\in\mathbb C$ large $n$ asymptotics for the biorthogonal polynomials $p_{n+k}^{(n,V)}(z)$ and $q_{n+k}^{(n,V)}(z)$ have been obtained in \cite{Claeys-Wang11} for strongly convex $V$. In our setting, these asymptotics thus also apply to the external field $\frac{1}{t}V$. Write $t=j/n$. In order to describe the asymptotics, uniform in $t$, for the polynomials $p_{j}^{(j,V/t)}(z)=p_{j}^{(n,V)}(z)$ and $q_{j}^{(j,V/t)}(z)=q_{j}^{(n,V)}(z)$ corresponding to the orthogonality weight $e^{-nV(x)}=e^{-\frac{j}{t}V(x)}$, we need to introduce some notations. 

Define (cf.~\cite[(1.33)]{Claeys-Wang11})
\begin{align}
  \gfn_t(z) \equiv {}& \int \log(z-s) d\mu_t(s), \label{eq:defn_gfn} \\
  \widetilde \gfn_t(z)\equiv {}& \int \log(e^z-e^s)d\mu_t(s), \label{eq:defn_gfn_tilde}
\end{align}
for $z\in\mathbb C\setminus(-\infty,b(t)]$, with the logarithms taking the principal value, such that
\begin{gather}
  \Re \gfn_{t, +}(x) = \Re \gfn_{t, -}(x), \quad \Re \widetilde{\gfn}_{t, +}(x) = \Re \widetilde{\gfn}_{t, -}(x), \qquad x\in (-\infty, b(t)), \label{eq:properties_of_g_func_at_R} \\
  -\Im \gfn_{t, +}(x) = \Im \gfn_{t, -}(x) = -\Im \widetilde{\gfn}_{t, +}(x) = \Im \widetilde{\gfn}_{t, -}(x) =
  \begin{cases}
    \pi \int^{b(t)}_x d\mu_t(s), & x \in (a(t), b(t)), \\
    \pi, & x \in (-\infty, a(t)].
  \end{cases}
\end{gather}
Thus the function
\begin{equation}\label{def Ft}
  F_t(z) := \frac{t}{2}(\widetilde{\gfn}_t(z) - \gfn_t(z)), \quad z \in \stripS = \{z \in \compC \mid \lvert \Im z \rvert < \pi \}
\end{equation}
is an analytic function on $\stripS$. We note that $F_1(x)$ for $x \in \realR$ agrees with $F(x)$ in \eqref{eq:defn_F}. We have the Euler-Lagrange variational conditions (cf.~\cite[(1.19) and (1.21)]{Claeys-Wang11})
\begin{align}
  \gfn_{t,\pm}(x) + \widetilde{\gfn}_{t,\mp}(x) - \frac{1}{t}V(x) - \ell_t = {}& 0, & x \in {}& [a(t), b(t)], \label{var eq} \\
  \gfn_{t,\pm}(x) + \widetilde{\gfn}_{t,\mp}(x) - \frac{1}{t}V(x) - \ell_t < {}& 0, & x \in {}& \mathbb R \setminus[a(t),b(t)], \label{var ineq}
\end{align}
where $\ell_t$ is defined such that for any $y \in [a(t), b(t)]$,
\begin{equation}
  \ell_t = \int \log \lvert x - y \rvert d\mu_t(x) + \int \log \lvert e^x - e^y \rvert d\mu_t(x) - \frac{V(y)}{t}.
\end{equation}

In terms of $c_0(t)$, $c_1(t)$ defined by \eqref{intro mod eq 1} and \eqref{intro mod eq 2}, we further define
\begin{equation}\label{def Gk}
  G_{t, k}(s) := c_1(t)^k \frac{(s + \frac{1}{2})(s - \frac{1}{2})^k}{\sqrt{s^2 - \frac{1}{4} - \frac{1}{c_1(t)}}}, \quad \hat{G}_{t, k}(s) := i\frac{e^{k(\frac{c_1(t)}{2} + c_0(t))}}{\sqrt{c_1(t)}} \frac{(s - \frac{1}{2})^{-k}}{\sqrt{s^2 - \frac{1}{4} - \frac{1}{c_1(t)}}},
\end{equation}
where the square root $\sqrt{s^2 - 1/4 - 1/c_1(t)}$ has its branch cut along $\gamma_1$ in $G_{t, k}(s)$, along $\gamma_2$ in $\hat{G}_{t, k}(s)$, and $\sqrt{s^2 - 1/4 - 1/c_1(t)} \sim s$ as $s \to \infty$ in both cases. Then we define
\begin{align}
  r_{t, k}(x) := {}& 2\lvert G_{t, k}(\Jinv_+(x)) \rvert, & \theta_{t, k}(x) := \arg(G_{t, k}(\Jinv_+(x))), \\
  \hat{r}_{t, k}(x) := {}& 2\lvert \hat{G}_{t, k}(\Jinv_-(x)) \rvert, & \hat{\theta}_{t, k}(x) := \arg(\hat{G}_{t, k}(\Jinv_-(x))).
\end{align}

Below, we summarize the asymptotic results for $p^{(n)}_j(x)$ and $q^{(n)}_j(e^x)$ that are needed in this paper. We give the reference of each result in \cite{Claeys-Wang11}. To simplify the notations, we denote
\begin{equation} \label{eq:defn_widetilde_pq}
  \widetilde{p}^{(n)}_j(x) = e^{-\frac{j}{2} \ell_{j/n}} e^{-\frac{nV(x)}{2}} p^{(n)}_j(x), \quad \widetilde{q}^{(n)}_j(x) = \frac{e^{\frac{j}{2} \ell_{j/n}}}{h^{(n)}_j} e^{-\frac{nV(x)}{2}} q^{(n)}_j(e^x).
\end{equation}
\begin{lem} \label{lem:asy_collected}
  \begin{enumerate}[label=(\alph*)]
  \item (\cite[Theorem 3(e)]{Claeys-Wang11} Asymptotics for the norming constants) \label{enu:lem:asy:1} As $n\to\infty$, we have
    \begin{equation}\label{as h}
      h_j^{(n)}=2\pi c_1(t)^{1/2}e^{j\ell_{t}}(1+\bigO(n^{-1})),
    \end{equation}
    uniformly for  $t=j/n$ in any compact $K\subset(0,+\infty)$.
    
  \item (\cite[Theorem 3(a)]{Claeys-Wang11} Outer asymptotics on the real line) \label{enu:lem:asy:2}
    For any $\epsilon>0$ and for any compact $K\subset(0,+\infty)$, there exist $n_0\in\mathbb N$ and a compact $U_{\epsilon} \subset \realR \setminus \{ 0 \}$ such that for any $n\geq n_0$, $t=j/n\in K$ and $x \in \mathbb R \setminus [a(t) - \epsilon, b(t) + \epsilon]$, we have
    \begin{equation}
      \frac{e^{nF_{t}(x)} \widetilde{p}_j^{(n)}(x)}{e^{\frac{j}{2} (\widetilde{\gfn}_{t}(x) + \gfn_{t}(x) - \frac{n}{j}V(x) - \ell_{t})}} \in U_{\epsilon}, \quad \frac{e^{-nF_{t}(x)} \widetilde{q}_j^{(n)}(x)}{e^{\frac{j}{2} (\widetilde{\gfn}_{t}(x) + \gfn_{t}(x) - \frac{n}{j}V(x) - \ell_{t})}} \in U_{\epsilon}.
    \end{equation}
  \item (\cite[Theorem 3(c) and (d)]{Claeys-Wang11} Microscopic scaling limit at the edge) \label{enu:lem:asy:3}
 Let $x=b(t) + u/(\pi \beta(t) j)^{2/3}$ and $y = a(t) - v/(\pi \alpha(t)j)^{2/3}$. For any compacts $K\subset(0,+\infty)$ and $U\subset\mathbb R$, we have as $n\to\infty$ that
    \begin{align} 
      \frac{1}{(\pi \beta(t) j)^{1/6}} e^{nF_{t}(x)} \widetilde{p}_j^{(n)}(x) = {}& \frac{\sqrt{2\pi}}{c_1(t)^{\frac{1}{2}}\left(s_b(t) - \frac{1}{2} \right) s_b(t)^{1/4}} \Ai(u) + \bigO(n^{-\frac{1}{3}}), \label{pj Airy} \\
      \frac{1}{(\pi \beta(t)j)^{1/6}} e^{-nF_{t}(x)} \widetilde{q}_j^{(n)}(x) = {}& \frac{1}{\sqrt{2\pi} c_1(t)^{\frac{1}{2}} s_b(t)^{1/4}} \Ai(u) + \bigO(n^{-\frac{1}{3}}),\label{qj Airy} \\
      \frac{1}{(\pi \alpha(\frac{j}{n}) j)^{1/6}} e^{nF_{t}(y)} \widetilde{p}_j^{(n)}(y) = {}& \frac{(-1)^j \sqrt{2\pi}}{c_1(t)^{\frac{1}{2}}\left(s_b(t) + \frac{1}{2} \right) s_b(t)^{1/4}} \Ai(v) + \bigO(n^{-\frac{1}{3}}), \label{pj Airy2} \\
      \frac{1}{(\pi \alpha(t)j)^{1/6}} e^{-nF_{t}(y)} \widetilde{q}_j^{(n)}(y) = {}& \frac{(-1)^j}{\sqrt{2\pi} c_1(t)^{\frac{1}{2}} s_b(t)^{1/4}} \Ai(v) + \bigO(n^{-\frac{1}{3}}),\label{qj Airy2}
    \end{align}
    uniformly in $t=j/n\in K$ and $u, v\in U$.
    
  \item (\cite[Theorem 3(c) and (d)]{Claeys-Wang11} Upper bound at the edge) \label{enu:lem:asy:4}
    For any compact $K\subset(0,+\infty)$, there exist constants $\epsilon, c>0$ such that
    \begin{align}
      e^{nF_{t}(x)} \widetilde{p}_j^{(n)}(x) = {}& \bigO(n^{\frac{1}{6}} e^{-cn^{2/3}(x-b(t))}), & e^{-nF_{t}(x)} \widetilde{q}_j^{(n)}(x) = {}& \bigO(n^{\frac{1}{6}} e^{-cn^{2/3}(x-b(t))}), \label{pj edge est} \\ % \label{qj edge est} 
      e^{nF_{t}(y)} \widetilde{p}_j^{(n)}(y) = {}& \bigO(n^{\frac{1}{6}} e^{-cn^{2/3}(a(t) - y)}), & e^{-nF_{t}(y)} \widetilde{q}_j^{(n)}(y) = {}& \bigO(n^{\frac{1}{6}} e^{-cn^{2/3}(a(t) - y)}),
    \end{align}
    as $n\to\infty$, uniformly in $t=j/n\in K$ and in $x\in[b(t)-\epsilon,b(t)+\epsilon], y\in[a(t)-\epsilon,a(t)+\epsilon]$. These are not sharp estimates, particularly not for $x<b(t)$, $y>a(t)$.
  \item (\cite[Theorem 3(c) and (d)]{Claeys-Wang11} Edge asymptotics upper bound, improvement for $x<b(t)$ and $y > a(t)$) \label{enu:lem:asy:4a}
    For any compact $K\subset(0,+\infty)$, there exist constants $\epsilon, c>0$ such that
    \begin{align}
      e^{nF_{t}(x)} \widetilde{p}_j^{(n)}(x) = {}& \bigO((b(t) - x + n^{-\frac{2}{3}})^{-\frac{1}{4}}), & e^{-nF_{t}(x)} \widetilde{q}_j^{(n)}(x) = {}& \bigO((b(t) - x + n^{-\frac{2}{3}})^{-\frac{1}{4}}), \label{pj edge est2} \\ % \label{qj edge est2} 
      e^{nF_{t}(y)} \widetilde{p}_j^{(n)}(y) = {}& \bigO((y - a(t) + n^{-\frac{2}{3}})^{-\frac{1}{4}}), & e^{-nF_{t}(y)} \widetilde{q}_j^{(n)}(y) = {}& \bigO((y - a(t) + n^{-\frac{2}{3}})^{-\frac{1}{4}}),
    \end{align}
    as $n\to\infty$, uniformly in $t=j/n\in K$ and in $x\in [b(t) - \epsilon, b(t)]$ and $y \in [a(t), a(t) + \epsilon]$.  These are not sharp estimates.
  \item (\cite[Theorem 3(b)]{Claeys-Wang11} Bulk asymptotics) \label{enu:lem:asy:5}
    Let $x = x^* + u/(\pi \varphi_t(x^*)n)$. For any $\epsilon>0$ and for any compacts $K\subset(0,+\infty)$, $U\subset\mathbb C$, we have uniformly in $x^* \in (a(t) + \epsilon, b(t) - \epsilon)$, $u\in U$, and $t=j/N\in K$ that
    \begin{equation} \label{eq:bulk_p}
      e^{nF_{t}(x)} \widetilde{p}^{(n)}_{j + k}(x) = r_{t, k}(x^*) \left[ \cos \left( n\pi \int^{b(t)}_{x^*} d\mu_{t}(\xi) + \theta_{t, k}(x^*) - u \right) + \bigO(n^{-1}) \right],
    \end{equation}
    \begin{multline} \label{eq:bulk_q}
      e^{-nF_{t}(x)} \widetilde{q}^{(n)}_{j + k}(x) = \\
      \frac{\hat{r}_{t, k}(x^*)}{\sqrt{2\pi} c_1(t)^{k + \frac{1}{2}} e^{k(\frac{c_1(t)}{2} + c_0)}} \left[ \cos \left( n\pi \int^{b(t)}_{x^*} d\mu_{t}(\xi) + \hat{\theta}_{t, k}(x^*) - u \right) + \bigO(n^{-1}) \right],
    \end{multline}
    as $n\to\infty$.
  \item (\cite[Theorem 3(a)]{Claeys-Wang11} Another outer asymptotics)
    Let $k\in\mathbb Z$ be fixed. For any compact $U\subset\compC \setminus [a(1),b(1)]$, we have (with $c_1 = c_1(1)$, $\gfn = \gfn_1$, $G_k = G_{1, k}$, $\widetilde{G}_{-k} = \widetilde{G}_{1, -k}$, $\ell = \ell_t$, and for all quantities depending implicitly on $t$, it is assumed that $t = 1$)
    \begin{align}
      p_{n + k}^{(n)}(z) = {}& G_k(\Jinv_1(z)) e^{n\gfn(z)} (1 + \bigO(n^{-1}) = c^k_1 \frac{(s + \frac{1}{2})(s - \frac{1}{2})^k}{\sqrt{s^2 - \frac{1}{4} - \frac{1}{c_1}}} e^{n\gfn(z)} (1 + \bigO(n^{-1})), \label{eq:asy_of_p_n+k} \\
      Cq_{n - k}^{(n)}(z) = {}& \widetilde{G}_{-k}(\Jinv_1(z)) e^{-n\gfn(z) + n\ell} (1 + \bigO(n^{-1})) \notag \\
                             &\qquad\qquad  = \frac{i}{\sqrt{c_1}} \frac{e^{-k(\frac{c_1}{2} + c_0)} (s - \frac{1}{2})^k}{\sqrt{s^2 - \frac{1}{4} - \frac{1}{c_1}}} e^{-n\gfn(z) + n\ell} (1 + \bigO(n^{-1})) \label{eq:asy_of_q_n-j}, \\
      h_{n - k}^{(n)} = {}& 2\pi c^{-k + \frac{1}{2}}_1 e^{-j(\frac{c_1}{2} + c_0)} e^{n\ell} (1 + \bigO(n^{-1})),\label{as hnk}
    \end{align}
    as $n\to\infty$, uniformly in $z\in U$, where we denote $s = \Jinv_1(z)$, and $\Jinv_1$ is an inverse function of $\Jlike_{c_1(1), c_0(1)}$ as defined in \cite[(1.26)]{Claeys-Wang11}. Here $Cq^{(n)}_j$ is the Cauchy transform of $q^{(n)}_j$
    \begin{equation} \label{eq:Cauchy_trans}
      Cq^{(n)}_j(z) = \frac{-1}{2\pi i z} \int^{\infty}_{-\infty} \frac{q^{(n)}_j(e^s)}{1 - s/z} e^{-nV(s)} ds.
    \end{equation}
  \end{enumerate}
\end{lem}

\begin{lemma} \label{enu:lemma zeros_2}
  For any $\epsilon > 0$, there is a $\delta > 0$ such that for all $n$ large enough, the zeros of $p^{(n)}_j(x)$ and $q^{(n)}_j(e^x)$ lie in the interval $(x_{\min} - \epsilon, x_{\min} + \epsilon)$ for all $j \leq \delta n$.
\end{lemma}
\begin{proof}
  By part \ref{enu:lemma zeros_1a} of Lemma \ref{lemma zeros}, for any $\epsilon>0$, there exists $\delta>0$ such that for $t\in (0,\delta)$, the support of $\mu_t$ is strictly included in $(x_{\min}-\epsilon/2, x_{\min}+\epsilon/2)$. From the asymptotics of the biorthogonal polynomials $p_k^{(n)}$ and $q_k^{(n)}$ in Lemma \ref{lem:asy_collected}, especially part \ref{enu:lem:asy:2}, it follows that for $n,k$ large and $k/n$ sufficiently close to $\delta$, all the zeros of $p_k^{(n)}(x)$ and of $q_k^{(n)}(e^x)$ lie in $[x_{\min}-\epsilon,x_{\min}+\epsilon]$. Moreover, like the zeros of orthogonal polynomials, the zeros of $p_k^{(n)}(x)$ interlace, as well as those of $q_k^{(n)}(e^x)$, which means that the zeros of $p_j^{(n)}(x)$ and of $q_j^{(n)}(e^x)$ lie in $[x_{\min}-\epsilon,x_{\min}+\epsilon]$ for any $j\leq k$. The interlacing property follows from the fact that the biorthogonal polynomials form an AT system \cite[Section 4.4]{Nikishin-Sorokin91}.
\end{proof}

Next we give rough bounds for the biorthogonal polynomials with small degrees and the norming constants, obtained by elementary estimates based on the zeros of the polynomials. 

\begin{lem} \label{lem:zeros_additional}
  Let $\epsilon > 0$, and let $\delta > 0$ be associated to $\epsilon$ as in Lemma \ref{enu:lemma zeros_2}. Suppose $n$ is large and $0 \leq j < \delta n$.
  \begin{enumerate}[label=(\alph*)]
  \item (Upper bound for multiple orthogonal polynomials of lower degrees) \label{enu:lem:asy:6}
    \begin{align}
      \lvert p^{(n)}_j(x) \rvert \leq {}&
      \begin{cases}
        (x - (x_{\min} - \epsilon))^j, & \text{$x \geq x_{\min} + \epsilon$}, \\
        ((x_{\min} + \epsilon) - x)^j, & \text{$x \leq x_{\min} - \epsilon$}, \\
        (2\epsilon)^j, & \text{$x \in [x_{\min} - \epsilon, x_{\min} + \epsilon]$},
      \end{cases} \label{eq:p_lower_degree} \\
      \lvert q^{(n)}_j(x) \rvert \leq {}&
      \begin{cases}
        (e^x - e^{x_{\min} - \epsilon})^j, & \text{$x \geq x_{\min} + \epsilon$}, \\
        (e^{x_{\min} + \epsilon} - e^x)^j, & \text{$x \leq x_{\min} - \epsilon$}, \\
        (e^{x_{\min} + \epsilon} - e^{x_{\min} - \epsilon})^j, & \text{$x \in [x_{\min} - \epsilon, x_{\min} + \epsilon]$}.
      \end{cases} \label{eq:q_lower_degree}
    \end{align}
  \item (Lower bound for the norming constants of lower degrees) \label{enu:lem:asy:7}
    \begin{equation} \label{eq:h_lower_degree}
      h^{(n)}_j \geq e^{-n(V(x_{\min}) + C(\epsilon))},
    \end{equation}
    where $C(\epsilon) > 0$ depends on $V$ but not on $j, n$, and $C(\epsilon) \to 0$ as $\epsilon \to 0$.
  \end{enumerate}
\end{lem}
\begin{proof}
  Part \ref{enu:lem:asy:6} is a direct consequence of Lemma \ref{enu:lemma zeros_2}. To prove part \ref{enu:lem:asy:7}, we note that by the orthogonality \eqref{kernel},
  \begin{equation}\label{hnj}
    h^{(n)}_j = \int_{\realR} p(x) q^{(n)}_j(e^x) e^{-nV(x)} dx,
  \end{equation}
  for any monic polynomial $p(x)$ of degree $j$. Assume that the zeros of $q^{(n)}_j(e^x)$ are $x_1, x_2, \dotsc, x_j \in (x_{\min} - \epsilon, x_{\min} + \epsilon)$. We let $p(x) = \prod^{j}_{i = 1} (x - x_i)$, such that the integrand in \eqref{hnj} is positive. We then have
  \begin{equation} \label{eq:lower_bound_h}
    \begin{split}
      h^{(n)}_j = {}& \int_{\realR} \left( \prod^j_{i = 1} (x - x_i) \right) \left( \prod^j_{i = 1} (e^x - e^{x_i}) \right)e^{-nV(x)} dx \\
      > {}& \int_{x_{\min} + 2\epsilon}^{x_{\min}+3\epsilon} (x - (x_{\min} + \epsilon))^j (e^x - e^{x_{\min} + \epsilon})^j e^{-n V(x)} dx,
    \end{split}
  \end{equation}
  
  Noting that $e^{-nV(x)}$ attains its minimum on the interval $[x_{\min} + 2\epsilon, x_{\min}+3\epsilon]$ at $x_{\min} + 3\epsilon$ due to the strong convexity, we obtain 
  \[h^{(n)}_j>\epsilon^j(e^{x_{\min}+2\epsilon}-e^{x_{\min}+\epsilon})^je^{-nV(x_{\min}+3\epsilon)}.
    \]
 The desired result now follows by a standard estimate of the right-hand side in the above equation.
\end{proof}

\subsection{Useful corollaries}

First, let $F_t$ be defined by \eqref{def Ft}. Since $F_t(z)$ is analytic with respect to $z$ on $\stripS$, the derivative $\frac{\partial}{\partial z} F_t(z)$ exists and is continuous there. On the other hand, for $z \in \stripS \setminus [a(t), b(t)]$, by the definition of $F_t(z)$ and the formula \eqref{eq:density_psi_t} for the density of $\mu_t$, we know that $\frac{\partial}{\partial t} F_t(z)$ also exists and is continuous there. Furthermore, if $w \in [a(t), b(t)]$ or in the vinicity of this interval, we have 
\begin{equation}
  \frac{\partial}{\partial t} F_t(w) = \frac{1}{2\pi i} \oint \frac{\partial}{\partial t} F_t(z) \frac{dz}{z - w},
\end{equation}
where the contour integral is over a closed contour within $\stripS$ that encloses $[a(t), b(t)]$. In this way, we have that the derivative $\frac{\partial}{\partial t} F_t(z)$ exists and is continuous on $\stripS$. Therefore, we have the following result:

\begin{cor}
Let $K_1 \subset \stripS$ and $K_2 \subset (0, \infty)$ be two compact sets. Then there exists $C_0 > 0$ such that for all $u, v \in K_1$ and $s, t \in K_2$, we have the estimates
  \begin{equation} \label{eq:continuity_of_F}
    \lvert F_t(u) - F_t(v) \rvert < C_0 \lvert u - v \rvert, \qquad \lvert F_s(u) - F_t(u) \rvert < C_0 \lvert s-t \rvert
  \end{equation}
  and
  \begin{equation} \label{eq:continuity_of_F_second}
    \lvert F_s(u) - F_s(v) - F_t(u) + F_t(v) \rvert < C_0 \lvert u - v \rvert \lvert s - t \rvert.
  \end{equation}
\end{cor}

Next, we state a corollary of Lemmas \ref{lem:asy_collected} and \ref{lem:zeros_additional}. We use the same notation as in these lemmas.
\begin{cor}
  Let $C_2>0$ be a large enough constant. For any $M>0$, there exist constants $C_1, C_3$ (depending on $C_2$ but not on $n$) such that the following inequalities hold for $n$ sufficiently large and $0\leq k/n<M$.
  \begin{enumerate}[label=(\alph*)]
  \item \label{cor:2a}
    We have
    \begin{equation} \label{eq:outer_estimate_orthogonals}
      \int_{\realR \setminus [-C_2, C_2]} \lvert \widetilde{p}^{(n)}_k(x) \rvert^2 (e^x + e^{-x}) dx < e^{-C_1n}, \quad \int_{\realR \setminus [-C_2, C_2]} \lvert \widetilde{q}^{(n)}_k(x) \rvert^2 (e^x + e^{-x}) dx < e^{-C_1 n}.
    \end{equation}  
  \item \label{cor:2b}
    For any $z = x + iy$ such that $x \in \realR$ and $y \in (-n^{-1}C_2, n^{-1}C_2)$, with $F(z) = F_1(z)$ defined in \eqref{eq:defn_F} and \eqref{def Ft},
    \begin{equation} \label{eq:rough_bound_of_summand}
      \lvert e^{nF(z)} \widetilde{p}^{(n)}_k(z) \rvert < e^{\frac{1}{3} C_1 n}, \quad  \lvert e^{-nF(z)} \widetilde{q}^{(n)}_k(z) \rvert < e^{\frac{1}{3} C_1 n}.
    \end{equation}
  \item \label{cor:2c}
    We have
    \begin{equation} \label{eq:inner_estimate_orthogonals}
      \int_{[-C_2, C_2]} \lvert \widetilde{p}^{(n)}_k(x) \rvert^2 dx < e^{C_3 n}, \quad \int_{[-C_2, C_2]} \lvert \widetilde{q}^{(n)}_k(x) \rvert^2 dx < e^{C_3 n}.
    \end{equation}
  \item \label{cor:2d}
    For any $\delta>0$,  $\delta<k/n<M$, there exists a constant $C_4>0$ depending on $C_2$ such that  
    \begin{equation} \label{eq:estimate_of_essential_part}
      \int_{[-C_2, C_2]} \lvert e^{nF_{k/n}(x)} \widetilde{p}^{(n)}_k(x) \rvert^2 dx < C_4, \quad \int_{[-C_2, C_2]} \lvert e^{-nF_{k/n}(x)} \widetilde{q}^{(n)}_k(x) \rvert^2 dx < C_4.
    \end{equation}
  \end{enumerate}
\end{cor}
\begin{proof}
  The  estimates in parts \ref{cor:2a} to \ref{cor:2c} are consequences of Lemma \ref{lem:asy_collected} for $\delta n\leq k<Mn$, and of Lemma \ref{lem:zeros_additional} for $0\leq k<\delta n$, with $\delta>0$ sufficiently small. They are all straightforward to prove. Estimate \ref{cor:2d} follows from the bulk asymptotics and the two upper bounds for the edge asymptotics in Lemma \ref{lem:asy_collected}.
\end{proof}

\section{Airy kernel}\label{section: Airy}

In this section we prove the Airy kernel limit \eqref{limAiryb} near the right endpoint $b$. The Airy kernel limit \eqref{limAirya} near the left endpoint can be proved in parallel, or alternatively we can use the symmetry of the model with respect to $V(x) \mapsto V(- x)+\frac{n-1}{n}x$ to map the left edge to the right. We omit the details.

In the rest of the section, we will need the auxiliary quantities $a(t)$, $b(t)$, $c_1(t)$, $c_0(t)$, $\alpha(t)$, $\beta(t)$ related to the equilibrium measure $\mu_t$ for $0<t\leq 1$. Whenever we write $a$, $b$, $c_1$, $c_0$, $\alpha$, $\beta$ without indicating the $t$-dependence, we refer to their values $a(1)$, $b(1)$, $c_1(1)$, $c_0(1)$, $\alpha(1)$, $\beta(1)$ for $t=1$. This notation is consistent with the one used in the introduction to state the results. We also remind the reader that $F(u)$ defined in \eqref{eq:defn_F} is equal to the $t = 1$ case of $F_t(u)$ defined in \eqref{def Ft}.

In this section we denote
\begin{equation} \label{eq:defn_x_and_y}
  u = b + \frac{\xi}{(\pi\beta n)^{2/3}}, \quad v = b + \frac{\eta}{(\pi\beta n)^{2/3}},
\end{equation}
where $\xi$ and $\eta$ are in a compact subset $U\subset \realR$.

In order to study large $n$ asymptotics, we split the correlation kernel $K_n(u,v)$ defined in \eqref{kernel} in $4$ parts:
\begin{equation}
K_n=K_n^{(1)}+K_n^{(2)}+K_n^{(3)}+K_n^{(4)},
\end{equation}
where, with $\widetilde{p}_j$ and $\widetilde{q}_j$ defined in \eqref{eq:defn_widetilde_pq},
\begin{align}
  K_n^{(1)} (u, v) = {}& \sum^{\lfloor \delta n\rfloor}_{j=0} \widetilde{p}_j^{(n)}(u) \widetilde{q}_j^{(n)}(v), & K_n^{(2)} (u, v) = {}& \sum^{\lfloor(1 - \delta')n\rfloor}_{j=\lfloor \delta n\rfloor +1} \widetilde{p}_j^{(n)}(u) \widetilde{q}_j^{(n)}(v), \\ % \label{K2}
  K_n^{(3)} (u, v) = {}& \sum^{\lfloor \left(1-Mn^{-2/3}\right)n\rfloor}_{j=\lfloor (1 - \delta')n\rfloor + 1} \widetilde{p}_j^{(n)}(u) \widetilde{q}_j^{(n)}(v), & K_n^{(4)} (u, v) = {}& \sum^{n-1}_{j=\lfloor \left(1-Mn^{-2/3}\right)n\rfloor +1} \widetilde{p}_j^{(n)}(u) \widetilde{q}_j^{(n)}(v).
\end{align}
Here we choose $\delta$ and $\delta'$ to be sufficiently small positive numbers whose values will be fixed in the proof of Lemma \ref{lemma K1}. $M$ can be any sufficiently large number independent of $n$. The proof of \eqref{limAiryb} is based on the following technical lemmas, which we will prove later in this section.
\begin{lemma} \label{lemma K1}
  There exists a constant $c>0$ such that for $n$ large enough
  \begin{equation} \label{estimateK1}
    \lvert e^{nF(u)} K_n^{(1)}(u, v) e^{-nF(v)} \rvert < e^{-cn},
  \end{equation}
  and
  \begin{equation}\label{estimateK2}
    \lvert e^{nF(u)} K_n^{(2)}(u, v) e^{-nF(v)} \rvert < e^{-c n}.
  \end{equation}
\end{lemma}

\begin{lemma}\label{lemma K3}
  There exist constants $c',c''>0$ independent of $M, n$ such that for $M, n$ large enough,
\begin{equation}\label{estimateK3}
  \left\lvert e^{nF(u)} K^{(3)}_n(u, v) e^{-nF(v)} \right\rvert \leq c'n^{2/3}e^{-c'' M}.
\end{equation}
\end{lemma}
\begin{lemma}\label{lemma K4}
  For any $M > 0$, 
 \begin{equation}\label{K Airy 2}
   \lim_{n\to\infty} \frac{1}{(\pi\beta n)^{2/3}} e^{nF(u)} K_n^{(4)}(u, v) e^{-nF(v)} = \int_0^{(\pi\beta)^{2/3} b'(1) M} \Ai(\xi + y)\Ai(\eta + y)dy.
 \end{equation}  
\end{lemma}

The Airy limit \eqref{limAiryb} follows from the lemmas above. By the identity \eqref{eq:AiryCD}, we have that for any $\epsilon > 0$, if $M$ is large enough, then
\begin{equation} \label{eq:ineq_Airy_1}
  \left\lvert \frac{\Ai(\xi)\Ai'(\eta)-\Ai'(\xi)\Ai(\eta)}{\xi - \eta} - \int_0^{(\pi\beta)^{2/3} b'(1) M} \Ai(\xi + y)\Ai(\eta + y)dy \right\rvert < \frac{\epsilon}{3}.
\end{equation}
By Lemma \ref{lemma K3}, if $M$ is large enough, then for all large enough $n$,
\begin{equation} \label{eq:ineq_Airy_2}
  \left\lvert \frac{1}{(\pi \beta n)^{2/3}} e^{nF(u)} K^{(3)}_n(u, v) e^{-nF(v)} \right\rvert \leq \frac{\epsilon}{3}.
\end{equation}
At last by Lemmas \ref{lemma K1} and \ref{lemma K4}, for $n$ large enough,
\begin{equation} \label{eq:ineq_Airy_3}
  \left\lvert \frac{1}{(\pi \beta n)^{2/3}} e^{nF(u)}\left( K_n^{(1)}(u, v) + K_n^{(2)}(u, v) \right)e^{-nF(v)} \right\rvert < \frac{\epsilon}{3}.
\end{equation}
The limit identity \eqref{K Airy 2}, the inequalities \eqref{eq:ineq_Airy_1}, \eqref{eq:ineq_Airy_2}, \eqref{eq:ineq_Airy_3} together with the arbitrariness of $\epsilon$ yield \eqref{limAiryb}.

\begin{proof}[Proof of Lemma \ref{lemma K1}]
  First we prove \eqref{estimateK1}. Assume that $\epsilon > 0$ is small enough such that $b > x_{\min} + 2\epsilon$ and $V(x_{\min}) + C(\epsilon) < V(b - \epsilon)$, where $C(\epsilon)$ is the constant in Lemma \ref{lem:zeros_additional}\ref{enu:lem:asy:7}. Then we take $\tilde{\delta}$ to be a small enough positive constant such that $a(\tilde{\delta})$ and $b(\tilde{\delta})$ are inside $(x_{\min} - \epsilon, x_{\min} + \epsilon)$. It is straightforward to check that we can choose $\delta<\tilde\delta$ positive and small enough such that we have the following inequalities for all $t \in [0, \delta]$ and $x, y \in [b - \epsilon, b + \epsilon]$:
  \begin{equation} \label{eq:estimate_K^1}
    \begin{split}
      & \frac{e^{-\frac{1}{2}V(x)} e^{-\frac{1}{2}V(y)}}{e^{-(V(x_{\min}) + C(\epsilon))}} (x - (x_{\min} - \epsilon))^t (e^y - e^{x_{\min} - \epsilon})^t \\
      \leq {}& \frac{e^{-V(b - \epsilon)}}{e^{-(V(x_{\min}) + C(\epsilon))}} ((b - \epsilon) - (x_{\min} - \epsilon))^{\delta} (e^{b - \epsilon} - e^{x_{\min} - \epsilon})^{\delta} \\
      < {}& e^{-C'},
    \end{split}
  \end{equation}
  where $C'$ is a positive constant. Using the estimates \eqref{eq:p_lower_degree}, \eqref{eq:q_lower_degree} and \eqref{eq:h_lower_degree}, we have that \eqref{eq:estimate_K^1} implies that for all $j < \delta n$,
  \begin{equation}
    \lvert \widetilde{p}_j^{(n)}(u) \widetilde{q}_j^{(n)}(v) \rvert = \lvert (h^{(n)}_j)^{-1} p_j^{(n)}(u) q_j^{(n)}(e^v) \rvert < e^{-C'n}.
  \end{equation}
  Finally using the estimate \eqref{eq:continuity_of_F} for the term $e^{n(F(u) - F(v))}$ and the fact that $|u-v|\leq C''n^{-2/3}$, we obtain 
  \begin{equation}
    \lvert \widetilde{p}_j^{(n)}(u) \widetilde{q}_j^{(n)}(v)e^{n(F(u)-F(v))} \rvert < e^{-(C'-\frac{C''}{n^{2/3}} C_0)n}.
  \end{equation}
  Taking the sum for $j$ from $0$ to $\lfloor \delta n\rfloor$, the inequality \eqref{estimateK1} follows with any $c <C'$ for $n$ sufficiently large.

  Next we prove \eqref{estimateK2}. For $\delta < j/n \leq 1 - \delta'$, there exists $\epsilon' > 0$ such that both $u$ and $v$ are greater than $b(j/n) + \epsilon'$, if $n$ is large enough. We can thus use the outer asymptotics in  Lemma \ref{lem:asy_collected}\ref{enu:lem:asy:2}. We obtain, uniformly in $t=j/n$,
  \begin{equation} \label{eq:est_K_2}
    \lvert e^{n(F_{t}(u) - F_{t}(v))} \widetilde{p}_j^{(n)}(u) \widetilde{q}_j^{(n)}(v) \rvert = 
    \bigO \left(e^{\frac{j}{2} (\widetilde{\gfn}_{t}(u) + \gfn_{t}(u) - \frac{1}{t}V(u) - \ell_{t})} e^{\frac{j}{2} (\widetilde{\gfn}_{t}(v) + \gfn_{t}(v) - \frac{1}{t}V(v) - \ell_{t})} \right),
  \end{equation}
  as $j,n\to\infty$.
  Using the inequality \eqref{var ineq}, we conclude that for all large enough $n$ and $\delta n < j \leq (1 - \delta')n$, there is $C''' > 0$ such that
  \begin{equation}
    \lvert e^{n(F_{t}(u) - F_{t}(v))} \widetilde{p}_j^{(n)}(u) \widetilde{p}_j^{(n)}(v) \rvert < e^{-C'''n}.
  \end{equation}
  Finally, by \eqref{eq:continuity_of_F_second}, for $n$ large enough,
  \begin{equation}
    e^{n(F_1(u) - F_1(v)) - n(F_{j/n}(u) - F_{j/n}(v))} < e^{nC_0(1 - \frac{j}{n}) \frac{\lvert \xi - \eta \rvert}{(\pi\beta n)^{2/3}}} < e^{\frac{C'''}{2} n},
  \end{equation}
  where $C_0$ is the same as in \eqref{eq:continuity_of_F_second}, and then
  \begin{equation}    
    \begin{split}
      & \lvert e^{nF(u)} K_n^{(2)}(u, v) e^{-nF(v)} \rvert  \\
      \leq {}& \sum^{\lfloor(1 - \delta')n\rfloor}_{j=\lfloor \delta n\rfloor +1} e^{n(F_1(u) - F_1(v)) - n(F_{j/n}(u) - F_{j/n}(v))} \lvert e^{n(F_{j/n}(u) - F_{j/n}(v))} \widetilde{p}_j^{(n)}(u) \widetilde{q}_j^{(n)}(v) \rvert \\
      < & ne^{-\frac{C'''}{2} n}.
    \end{split}
  \end{equation}
  Thus we prove \eqref{estimateK2} for any $c < C'''/2$ if $n$ is sufficiently large.
\end{proof}

\begin{proof}[Proof of Lemma \ref{lemma K3}]
  For all $t=j/n \in [1 - \delta', 1]$, we use the upper bound at the edge in Lemma \ref{lem:asy_collected}\ref{enu:lem:asy:4}. Together with  \eqref{eq:continuity_of_F_second}, we obtain
  \begin{equation} \label{eq:essential_step_K^3}
    \begin{split}
       &\lvert e^{nF(u)} \widetilde{p}_j(u) \widetilde{q}_j(v) e^{-nF(v)} \rvert = {} e^{n(F_1(u) - F_1(v) - F_{t}(u) + F_{t}(v))} \lvert e^{nF_{t}(u)} \widetilde{p}_j(u) \widetilde{q}_j(v) e^{-nF_{t}(v)} \rvert \\
       &\qquad\qquad\leq {} C'n^{1/3}\exp \left( \frac{C_0 n^{\frac{1}{3}}}{(\pi \beta)^{\frac{2}{3}}} \lvert \xi - \eta \rvert \lvert 1 - t \rvert \right) e^{-cn^{2/3}(u - b(j/n))} e^{-cn^{2/3}(v - b(j/n))}.
    \end{split}
  \end{equation}
  where $C_0$ is the constant in \eqref{eq:continuity_of_F_second} and $c$ is the constant in \eqref{pj edge est}. By Lemma \ref{lemma endpoints}, $b(t)$ is differentiable at $t=1$ with $b'(1)>0$, and this implies the existence of a constant $C'' < 0$ such that 
\begin{equation}
  \lvert b(1) - b(t) \rvert \geq C''(1 - t)
\end{equation}
for $t \in [1 - \delta', 1]$, where we assume that $\delta'$ is small enough. Hence we have the inequality
\begin{multline} \label{eq:one_term_in_Lemma_K_3}
  \lvert e^{nF(u)} \widetilde{p}_j(u) \widetilde{q}_j(v) e^{-nF(v)} \rvert \\
  \leq C' n^{1/3}\exp \left( \frac{C_0 n^{\frac{1}{3}}}{(\pi \beta)^{\frac{2}{3}}} \lvert \xi - \eta \rvert \lvert 1 - \frac{j}{n} \rvert \right) e^{-c \left( C'' n^{\frac{2}{3}} (1 - \frac{j}{n}) + \frac{\xi}{(\pi \beta)^{\frac{2}{3}}} \right)} e^{-c \left( C'' n^{\frac{2}{3}} (1 - \frac{j}{n}) + \frac{\eta}{(\pi \beta)^{\frac{2}{3}}} \right)}.
\end{multline}
Taking the sum for all $j$ from $\lfloor (1-\epsilon_2)n+1\rfloor$ to $\lfloor \left(1-Mn^{-2/3}\right)n\rfloor$, we prove the lemma with $c', c''$ in \eqref{estimateK3} determined by $c, C_0, C', C''$ in \eqref{eq:one_term_in_Lemma_K_3}.
\end{proof}

\begin{proof}[Proof of Lemma \ref{lemma K4}]
  By Taylor expansion, for all $j/n \in [1 - M n^{-2/3}, 1]$, $n\to\infty$,
  \begin{gather} \label{b expansion}
    b(j/n) = b - b'(1)\frac{n - j}{n} + \bigO(n^{-\frac{4}{3}}), \\
    \beta(j/n) = \beta + \bigO(n^{-\frac{2}{3}}), \quad c_1(j/n) = c_1 + \bigO(n^{-\frac{2}{3}}), \quad s_b(j/n) = s_b + \bigO(n^{-\frac{2}{3}}).
  \end{gather}
  The expressions \eqref{eq:defn_x_and_y} for $u$ and $v$ can be rewritten as
  \begin{equation}
    u = b(j/n) + \frac{\xi + \frac{n - j}{n^{\frac{1}{3}}} b'(1) (\pi \beta)^{\frac{2}{3}}}{(\pi \beta j)^{\frac{2}{3}}} + \bigO(n^{-\frac{4}{3}}), \quad v = b(j/n) + \frac{\eta + \frac{n - j}{n^{\frac{1}{3}}} b'(1) (\pi \beta)^{\frac{2}{3}}}{(\pi \beta j)^{\frac{2}{3}}} + \bigO(n^{-\frac{4}{3}}).
  \end{equation}
  By Lemma \ref{lem:asy_collected}\ref{enu:lem:asy:3}, we have
  \begin{equation*} \label{eq:essential_step_K^4}
    \begin{split}
      & e^{nF(u)} \widetilde{p}_j(u) \widetilde{q}_j(v) e^{-nF(v)} \\
      = {}& e^{n(F_1(u) - F_1(v)) - n(F_{j/n}(u) - F_{j/n}(v))} \left( e^{nF_{j/n}(u)} \widetilde{p}_j(u) \right) \left( e^{-nF_{j/n}(v)} \widetilde{q}_j(v) \right) \\
      = {}& e^{n(F_1(u) - F_1(v)) - n(F_{j/n}(u) - F_{j/n}(v))} \frac{ (\pi \beta(\frac{j}{n})j)^{\frac{1}{3}}}{c_1(\frac{j}{n}) (s_b(\frac{j}{n}) - \frac{1}{2}) s_b(\frac{j}{n})^{\frac{1}{2}}} \\
      & \times \Ai \left( \xi + \frac{n - j}{n^{\frac{1}{3}}} b'(1) (\pi \beta)^{\frac{2}{3}} + \bigO(n^{-\frac{2}{3}}) \right) \Ai \left( \eta + \frac{n - j}{n^{\frac{1}{3}}} b'(1) (\pi \beta)^{\frac{2}{3}} + \bigO(n^{-\frac{2}{3}}) \right) (1 + \bigO(n^{-\frac{1}{3}})).
      \end{split}
      \end{equation*}
     By \eqref{eq:continuity_of_F_second}, we have $e^{n(F_1(u) - F_1(v)) - n(F_{j/n}(u) - F_{j/n}(v))} = 1 + \bigO(n^{-1/3})$, and this implies 
       \begin{equation} \label{eq:essential_step_K^4-2}
    \begin{split}
      & e^{nF(u)} \widetilde{p}_j(u) \widetilde{q}_j(v) e^{-nF(v)} \\
      = {}& \frac{ (\pi \beta n)^{\frac{1}{3}}}{c_1 (s_b - \frac{1}{2}) s_b^{\frac{1}{2}}} \Ai \left( \xi + \frac{n - j}{n^{\frac{1}{3}}} b'(1) (\pi \beta)^{\frac{2}{3}} \right) \Ai \left( \eta + \frac{n - j}{n^{\frac{1}{3}}} b'(1) (\pi \beta)^{\frac{2}{3}} \right) (1 + \bigO(n^{-\frac{1}{3}})).
    \end{split}
  \end{equation}
 Taking the sum of the left-hand side of \eqref{eq:essential_step_K^4-2} with respect to $j$ from $\lfloor \left(1-Mn^{-2/3}\right)n\rfloor +1$ to $n - 1$ and replacing the sum by an integral with small correction, we have that
  \begin{multline}
    e^{nF(u)} K_n^{(4)}( u, v) e^{-nF(v)} = \\
    \frac{n^{\frac{2}{3}}}{c_1 (s_b - \frac{1}{2})s^{\frac{1}{2}}_b (\pi \beta)^{\frac{1}{3}} b'(1)} \int_0^{(\pi\beta)^{2/3} b'(1) M} \Ai(u+y)\Ai(v+y)dy (1 + \bigO(n^{-\frac{1}{3}})).
  \end{multline}
Using expression \eqref{a1b1der} for $b'(1)$ in terms of $s_b$ and \eqref{eq:s_b_in_c_10} and \eqref{def:ab} to obtain an expression for $s_b$ in terms of $c_1$, we obtain the result.
\end{proof}

\section{Sine kernel}\label{section: sine}

\subsection{Preliminaries}

In this section we denote
\begin{equation}\label{def uv}
  u = x^*+\frac{\xi}{\pi\psi(x^*)n}, \quad v = x^*+\frac{\eta}{\pi\psi(x^*)n},
\end{equation}
where $x^* \in (a, b)$, and $\xi, \eta$ are in a compact subset $U \subset \compC$. Although the kernel function $K^{(n)}_n(u, v)$ has a probabilistic meaning only if $u, v$ are real, for technical reasons to be explained below, we need to consider complex $\xi$ and $\eta$.

The goal of this section is to compute $e^{nF(u)} K^{(n)}_n(u, v) e^{-nF(v)}$ to prove \eqref{limsin}. Instead of computing it directly, we consider
\begin{multline} \label{eq:alternative_kernel_product}
  (\exp_{\lfloor \delta n \rfloor}(u) - \exp(v)) e^{nF(u)} K^{(n)}_n(u, v) e^{-nF(v)}\\ = e^{n(F(u) - F(v))} (\exp_{\lfloor \delta n \rfloor}(u) - \exp(v)) \sum^{n - 1}_{k = 0} \widetilde{p}^{(n)}_k(u) \widetilde{q}^{(n)}_k(v), 
\end{multline}
where $\delta$ is a small enough positive number and $\exp_k$ is the truncated exponential sum
\begin{equation}
  \exp_k(x) = \sum^{k}_{i = 0} \frac{x^i}{i!}.
\end{equation}

We need several estimates of $\exp_k(z)$. The first ones are very rudimentary: For $k, l \in \intZ_+$ and $z = x + iy$,
\begin{equation} \label{eq:rudimentary_estimate_exp_k_1}
  \lvert \exp_k(z) \rvert \leq e^x + e^{-x}, \quad \lvert \exp(z) - \exp_k(z) \rvert \leq e^x + e^{-x}, \quad \lvert \exp_k(z) - \exp_l(z) \rvert \leq e^x + e^{-x}.
\end{equation}
The next one can be shown using Stirling's formula. For any $C_2 > 0$, there exists a constant $C_5$, depending only on $C_2$, such that for all $k < l \in \intZ_+$
\begin{equation} \label{eq:Stirling_consequence}
  \max_{\lvert z \rvert < C_2} \lvert \exp(z) - \exp_k(z) \rvert < C_5 e^{-\frac{k}{2} \log(k)}, \quad \max_{\lvert z \rvert < C_2} \lvert \exp_l(z) - \exp_k(z) \rvert < C_5 e^{-\frac{k}{2} \log(k)}.
\end{equation}

Later in Section \ref{sec:practical_limit_Sine}, we will prove that for $\xi, \eta$ in a compact $K\subset\mathbb C$, uniformly
\begin{equation} \label{eq:practical_limit_Sine}
  \lim_{n \to \infty} (\exp_{\lfloor \delta n \rfloor}(u) - \exp(v)) e^{nF(u)} K^{(n)}_n(u, v) e^{-nF(v)} = \frac{e^{x^*}}{\pi} \sin\pi(\xi - \eta).
\end{equation}
With the help of \eqref{eq:Stirling_consequence}, when $\xi \neq \eta$, this is equivalent to
\begin{equation} \label{eq:limit_id_for_sine}
  \lim_{n \to \infty} \frac{e^{n(F(u) - F(v))}}{\pi\psi(x^*)n} K^{(n)}_n(u, v) = \lim_{n\to\infty} \frac{\sin\pi(\xi - \eta)}{\pi(e^{u - x^*} - e^{v - x^*})\pi\psi(x^*)n},
\end{equation}
which is, by Taylor approximations of $F(u)$, $F(v)$, $e^{u - x^*}$, $e^{v - x^*}$ at $x^*$, equivalent to the sine kernel limit \eqref{limsin}. Hence the $\xi \neq \eta$ part of \eqref{limsin} is proved. To prove the $\xi = \eta$ case of \eqref{limsin}, we note that functions on both sides of \eqref{eq:practical_limit_Sine} are analytic in $\xi$ and $\eta$. With $\xi$ fixed, for all $\eta$ on a small circle centred at $\xi$, the limit identity \eqref{eq:limit_id_for_sine} holds uniformly, so it extends to $\eta = \xi$. By this argument we complete the proof of \eqref{limsin}.

\subsection{Proof of \eqref{eq:practical_limit_Sine}} \label{sec:practical_limit_Sine}

We can expand $e^y\widetilde{q}^{(n)}_j(y)$ as a linear combination of $\widetilde q^{(n)}_0, \widetilde q^{(n)}_1, \ldots$,
\begin{equation}\label{exp q}
  \exp(y) \widetilde{q}^{(n)}_j(y) = \sum^{j + 1}_{k = 0} a_{j, k} \widetilde{q}^{(n)}_k(y), \quad a_{j, k} = \int \widetilde{p}^{(n)}_k(x) \widetilde{q}^{(n)}_j(x) e^x dx.
\end{equation}
On the other hand, we have
\begin{equation}\label{exp p}
  \exp_{\lfloor \delta n \rfloor}(x) \widetilde{p}^{(n)}_j(x) = \sum^{j + \lfloor \delta n \rfloor}_{k = 0} b_{j, k} \widetilde{p}^{(n)}_k(x), \quad b_{j, k} = \int \widetilde{p}^{(n)}_j(x) \widetilde{q}^{(n)}_k(x) \exp_{\lfloor \delta n \rfloor}(x) dx.
\end{equation}
Since the left hand side is a polynomial of degree $j+\lfloor \delta n\rfloor$, we see that $b_{j, k} = 0$ if $k > j + \lfloor \delta n \rfloor$; similarly  $a_{j,k}=0$ for $k>j+1$.

To compute the right-hand side of \eqref{eq:alternative_kernel_product}, we write
\begin{equation} \label{eq:tri_division}
  (\exp_{\lfloor \delta n \rfloor}(u) - \exp(v)) \sum^{n - 1}_{k = 0} \widetilde{p}^{(n)}_k(u) \widetilde{q}^{(n)}_k(v) = J_1(u, v) + J_2(u, v) - a_{n - 1, n} \widetilde{p}^{(n)}_{n - 1}(u) \widetilde{q}^{(n)}_n(v),
\end{equation}
where
\begin{equation}\label{def J1 J2}
  J_1(u, v) = \sum^{n - 1}_{j, k = 0} (b_{k, j} - a_{j, k}) \widetilde{p}^{(n)}_j(u) \widetilde{q}^{(n)}_k(v), \quad J_2(u, v) = \sum^{n - 1}_{j = n - \lfloor \delta n \rfloor} \sum^{j + \lfloor \delta n \rfloor}_{k = n} b_{j, k} \widetilde{p}^{(n)}_k(u) \widetilde{q}^{(n)}_j(v).
\end{equation}
Substituting \eqref{def J1 J2} in \eqref{eq:tri_division}, we get a double sum at the right hand side of \eqref{eq:tri_division}. This should be seen as the analogue of the Christoffel-Darboux formula \eqref{eq:CD} valid in the case where $f(\lambda)=\lambda$. In that case, the counterpart of the prefactor $(\exp_{\lfloor \delta n \rfloor}(u) - \exp(v))$ would be $(u-v)$, and only two terms survive on the right, by \eqref{eq:CD}. In our situation where $f(\lambda)=e^\lambda$, no exact cancellations take place on the right, but fortunately, it will turn out that many terms in the double sum on the right will be small as $n\to\infty$. Indeed, we will show that $J_1$ tends to $0$ as $n\to\infty$, and that the terms in $J_2$ for which $j-k$ is large and the ones for which $n-j$ or $k-n$ is large, have only a small contribution. The main contribution will come from the terms near $j=n-1$ and $k=n-1$ in the $(j,k)$-plane.

\begin{lemma}\label{lemma sine1}
Let $J_1$ be defined by (\ref{def J1 J2}), with $a_{j,k}, b_{j,k}$ defined by (\ref{exp q}) and (\ref{exp p}), and with $u,v$ as in (\ref{def uv}). We have
\begin{equation}\label{eq lemma sine 1}
\lim_{n\to\infty}e^{n(F(u)-F(v))}J_1(u,v)=0.
\end{equation}
\end{lemma}

\begin{lemma} \label{lemma sine2}
  Let $J_2$ be defined by \eqref{def J1 J2}, with $a_{j,k}, b_{j,k}$ defined by \eqref{exp q} and \eqref{exp p}, and with $u,v$ as in \eqref{def uv}. For any $\epsilon>0$, there exist $M_0, n_0$ such that for $M\geq M_0, n\geq n_0$, 
  \begin{equation} \label{eq lemma sine 2}
    \left\lvert e^{n(F(u)-F(v))}\left(J_2(u,v)-\sum_{j=n-M}^{n-1}\sum_{k=n}^{j+M}a_{k,j}\widetilde p^{(n)}_k(u)\widetilde q^{(n)}_j(v)\right)\right\rvert<\epsilon.
  \end{equation}
\end{lemma}

\begin{lemma} \label{lemma sine3}
  Let $a_{n + j,n - k}$ be defined by \eqref{exp q}. As $n\to\infty$, we have
  \begin{equation}\label{estimate a}
    a_{n + j, n - k}=
    \begin{cases}
      \alpha_{j + k} + \bigO(n^{-1}), & j + k \geq -1, \\
      0, & j + k \leq -2,
    \end{cases},
    \quad \text{where} \quad
    \alpha_l = 
    \begin{cases}
      \left( \frac{c_1}{(1 + l)!} + \frac{1}{l!} \right) e^{\frac{1}{2} c_1 + c_0}, & l \geq 0, \\
      c_1 e^{\frac{1}{2} c_1 + c_0}, & l = -1.
    \end{cases}
  \end{equation}
  Here $c_0=c_0(1)$, $c_1=c_1(1)$ are defined by \eqref{eq:first_def_of_c_0} and \eqref{eq:first_def_of_c_1}.
\end{lemma}

\begin{cor} \label{lemma sine4}
  Let $a_{k,j}$ be defined by \eqref{exp q}. For any $\epsilon>0$, there exist $M_0, n_0$ such that for $M\geq M_0, n\geq n_0$, we have
  \begin{equation}\label{eq lemma sine 4}
    \left\lvert e^{n(F(u)-F(v))}\left(\sum_{j=n-M}^{n-1}\sum_{k=n}^{j+M}a_{k,j}\widetilde p^{(n)}_k(u)\widetilde q^{(n)}_j(v) -a_{n-1,n}\widetilde p^{(n)}_{n-1}(u)\widetilde q^{(n)}_n(v)\right)-\frac{e^{x^*}}{\pi}\sin\pi(\xi-\eta)\right\rvert <\epsilon.
  \end{equation}
\end{cor}

Summing up \eqref{eq lemma sine 1}, \eqref{eq lemma sine 2}, and \eqref{eq lemma sine 4}, we obtain that \eqref{sec:practical_limit_Sine} holds. Hence it suffices to prove Lemma \ref{lemma sine1}, Lemma \ref{lemma sine2}, and Corollary \ref{lemma sine4}, as well as Lemma \ref{lemma sine3} on which Corollary \ref{lemma sine4} relies. We note that in the proofs below, $C_2$ is a large enough positive constant, and other constant terms are defined where they are used.

\begin{proof}[Proof of Lemma \ref{lemma sine1}]
  Using \eqref{exp q}, \eqref{exp p}, and the orthogonality of $\widetilde p^{(n)}_j$ and $\widetilde q^{(n)}_j$, we obtain for any $0 \leq j, k \leq n$,
  \begin{equation}
    \lvert b_{k, j} - a_{j, k} \rvert = \left\lvert \int_{\realR} (\exp(x) - \exp_{\lfloor \delta n \rfloor}(x)) \widetilde{p}^{(n)}_k(x) \widetilde{q}^{(n)}_j(x) dx\right\rvert.
  \end{equation}
  Using \eqref{eq:rudimentary_estimate_exp_k_1}, the Cauchy-Schwarz inequality, and \eqref{eq:outer_estimate_orthogonals}, we obtain
  \begin{equation} \label{eq:I_1_new}
    \begin{split}
      I_1 := {}& \left\lvert \int_{\realR \setminus [-C_2, C_2]} (\exp(x) - \exp_{\lfloor \delta n \rfloor}(x)) \widetilde{p}^{(n)}_k(x) \widetilde{q}^{(n)}_j(x) dx \right\rvert \\
      \leq {}& \left\lvert \int_{\realR \setminus [-C_2, C_2]} (e^x + e^{-x}) \widetilde{p}^{(n)}_k(x) \widetilde{q}^{(n)}_j(x) dx \right\rvert \\
      \leq {}& e^{-C_1n},
    \end{split}
  \end{equation}
  and by \eqref{eq:Stirling_consequence}, the Cauchy-Schwarz inequality, and \eqref{eq:inner_estimate_orthogonals},
  \begin{equation}
    \begin{split}
      I_2 := {}& \left\lvert \int_{[-C_2, C_2]} (\exp(x) - \exp_{\lfloor \delta n \rfloor}(x)) \widetilde{p}^{(n)}_k(x) \widetilde{q}^{(n)}_j(x) dx \right\rvert \\
      \leq {}& C_5 e^{-\frac{\lfloor \delta n \rfloor}{2} \log(\lfloor \delta n \rfloor)} \left\lvert \int_{[-C_2, C_2]} \widetilde{p}^{(n)}_k(x) \widetilde{q}^{(n)}_j(x) dx \right\rvert \\
      \leq {}& C_5 e^{-\frac{\delta n}{2} \log(\delta n) + C_3 n}.
    \end{split}
  \end{equation}
  Using the estimate \eqref{eq:rough_bound_of_summand} for $e^{n(F(u) - F(v))} \widetilde{p}^{(n)}_j(u) \widetilde{q}^{(n)}_k(v)$, we have that for each pair $j, k$,
  \begin{equation} \label{eq:useful_in_J_1_J_2}
    \begin{split}
      \left\lvert e^{n(F(u) - F(v))} (b_{k, j} - a_{j, k}) \widetilde{p}^{(n)}_j(u) \widetilde{q}^{(n)}_k(v) \right\rvert = {}& (I_1+I_2) \left\lvert e^{nF(u)} \widetilde{p}^{(n)}_j(u) e^{-nF(v))} \widetilde{q}^{(n)}_k(v) \right\rvert \\
      < {}& \left(e^{-C_1n} + C_5 e^{-\frac{\delta n}{2} \log(\delta n) + C_3 n}\right) e^{\frac{2C_1}{3} n},
    \end{split}
  \end{equation}
  and then finish the proof by
  \begin{equation}
  \lvert e^{n(F(u) - F(v))} J_1(u, v) \rvert < n^2(e^{-C_1 n} + C_5 e^{-\frac{\delta n}{2} \log(\delta n) + C_3 n}) e^{\frac{2C_1}{3} n} = o(1),\quad \text{as } n\to\infty.
  \end{equation}
\end{proof}

\begin{proof}[Proof of Lemma \ref{lemma sine2}]
  When we consider $J_2$ defined in \eqref{def J1 J2}, the indices $j, k$ satisfy $n - \lfloor \delta n \rfloor \leq j \leq n - 1$ and $n \leq k \leq j + \lfloor \delta n \rfloor$. If $m < k - j$, then by the biorthogonality,
  \begin{equation}
    \int_{\realR} \exp_m(x) \widetilde{p}^{(n)}_j(x) \widetilde{q}^{(n)}_k(x) dx = 0.
  \end{equation}
  Hence we have 
  \begin{equation} \label{eq:B_jk_in_J_2}
    \lvert b_{j, k} \rvert = \left\lvert \int_{\realR} [\exp_{\lfloor \delta n \rfloor}(x) - \exp_{k - j - 1}(x)] \widetilde{p}^{(n)}_j(x) \widetilde{q}^{(n)}_k(x) dx \right\rvert \leq I_3 + I_4,
  \end{equation}
  where
  \begin{align}
    I_3 := {}& \left\lvert \int_{\realR \setminus [-C_2, C_2]} [\exp_{\lfloor \delta n \rfloor}(x) - \exp_{k - j - 1}(x)] \widetilde{p}^{(n)}_j(x) \widetilde{q}^{(n)}_k(x) dx \right\rvert, \\
    I_4 := {}& \left\lvert \int_{[-C_2, C_2]} [\exp_{\lfloor \delta n \rfloor}(x) - \exp_{k - j - 1}(x)] \widetilde{p}^{(n)}_k(x) \widetilde{q}^{(n)}_j(x) dx \right\rvert.
  \end{align}
  Similar to \eqref{eq:I_1_new}, we have $I_3  < e^{-C_1n}$. On the other hand,
  \begin{multline} \label{eq:est_I_4}
    I_4 = \left\lvert \int_{[-C_2, C_2]} [\exp_{\lfloor \delta n \rfloor}(x) - \exp_{k - j - 1}(x)] e^{-n(F_{k/n}(x) - F_{j/n}(x))} \right. \\
    \left. \vphantom{\int_{[-C_2, C_2]}}
      \times e^{nF_{k/n}(x)} \widetilde{p}^{(n)}_k(x) e^{-nF_{j/n}(x)} \widetilde{q}^{(n)}_j(x) dx \right\rvert.
  \end{multline}
  Using the estimate \eqref{eq:Stirling_consequence} to estimate $\exp_{\lfloor \delta n \rfloor}(x) - \exp_{k - j - 1}(x)$, and then using \eqref{eq:continuity_of_F} to estimate $F_{k/n}(x) - F_{j/n}(x)$ and \eqref{eq:estimate_of_essential_part} to estimate $e^{nF_{k/n}(x)} \widetilde{p}_k(x) e^{-nF_{j/n}(x)} \widetilde{q}_j(x)$, we have
  \begin{equation}
    I_4 \leq C_4C_5 e^{-\frac{k - j - 1}{2} \log(k - j - 1) + C_0 (k-j)}.
  \end{equation}
  It follows that 
  \begin{equation}\label{estimate bjk}
    |b_{j,k}|\leq e^{-C_1n}+C_4C_5e^{-\frac{k - j - 1}{2} \log(k - j - 1) + C_0 (k-j)}.
  \end{equation}
  Now we estimate $e^{n(F(u) - F(v))} \widetilde{p}^{(n)}_k(u) \widetilde{q}^{(n)}_j(v)$. Since in the setting of Lemma \ref{lemma sine2}, $|j-n|, |k-n|<\lfloor \delta n\rfloor$ and $\delta$ small, we have that $x^*$ lies in $(a(j/n), b(j/n))$ and $(a(k/n), b(k/n))$, where $[a(t), b(t)]$ is the support of the equilibrium measure $\mu_t$, and it follows from the bulk asymptotics in Lemma \ref{lem:asy_collected}\ref{enu:lem:asy:5} that $e^{nF_{k/n}(u)}\widetilde p^{(n)}_k(u)$ and $e^{-nF_{j/n}(v)}\widetilde q^{(n)}_j(v)$ are bounded in $n$. Combining this with \eqref{eq:continuity_of_F}, there is a constant $C'$ such that
  \begin{equation} \label{eq:alt_est_pq}
    \begin{split}
      e^{n(F(u) - F(v))} \widetilde{p}^{(n)}_k(u) \widetilde{q}^{(n)}_j(v) = {}& e^{n(F_1(u) - F_{k/n}(u))} e^{-n(F_1(v) - F_{j/n}(v))} e^{nF_{k/n}(u)} \widetilde{p}^{(n)}_k(u) e^{-nF_{j/n}(v)} \widetilde{q}^{(n)}_j(v) \\
      \leq {}& e^{C_0(k - n)} e^{C_0(n - j)} C' \\
      = {}& C' e^{C_0(k-j)}.
    \end{split}
  \end{equation}
  Combining \eqref{estimate bjk} with the estimate \eqref{eq:alt_est_pq} for $e^{n(F(u) - F(v))} \widetilde{p}^{(n)}_k(u) \widetilde{q}^{(n)}_j(v)$, we have that if $n$ is large enough, then the value of $\lvert b_{j, k} e^{n(F(u) - F(v))} \widetilde{p}^{(n)}_k(u) \widetilde{q}^{(n)}_j(v) \rvert$ 
  \begin{itemize}
  \item 
    decays exponentially fast as $k - j$ increases, as long as $k - j \leq \sqrt{n}$, and
  \item
    is bounded by $e^{-\sqrt{n}}$ for $k - j \in (\sqrt{n}, \delta n]$, given that $\delta$ is small enough.
  \end{itemize}
  Thus, in the double sum defining $J_2$ in \eqref{def J1 J2}, the terms for $k-j$ large will give a small contribution in $e^{n(F(u) - F(v))} \widetilde{p}^{(n)}_k(u) \widetilde{q}^{(n)}_j(v)$: for any $\epsilon_1 > 0$, we have that
  \begin{equation}\label{eq:J2}
    \left\lvert e^{n(F(u) - F(v))} \left( J_2(u, v) - \sum^{n - 1}_{j = n - M} \sum^{j + M}_{k = n} b_{j, k} \widetilde{p}^{(n)}_k(u) \widetilde{q}^{(n)}_j(v) \right) \right\rvert < \epsilon_1,
  \end{equation}
  if $M$ and $n$ are large enough.
  
  We note that inequality \eqref{eq:useful_in_J_1_J_2} also holds when $k, j$ are around $n$. Hence for any finite $M$, there exists $\epsilon > 0$, such that
\begin{equation}
  \left\lvert e^{n(F(u) - F(v))} \sum^{n - 1}_{j = n - M} \sum^{j + M}_{k = n} (b_{j, k} - a_{k, j}) \widetilde{p}^{(n)}_k(u) \widetilde{q}^{(n)}_j(v) \right\rvert = \bigO(e^{-\epsilon n}), \quad n\to\infty.
\end{equation}
This, together with \eqref{eq:J2}, completes the proof.
\end{proof}

\begin{proof}[Proof of Lemma \ref{lemma sine3}]
  When $j + k \leq -2$ or $(n - k + 1) < (n + j)$, the result follows from the biorthogonality. Below we prove the other cases. Let $C_2>0$ be a sufficiently large constant such that the support of the equilibrium measure $\mu$ is contained in the interval $(-C_2, C_2)$. Using the exponential decay of the polynomials $\widetilde p^{(n)}_j$ and $\widetilde q_j^{(n)}$, see (\ref{eq:outer_estimate_orthogonals}), it is straightforward to approximate $a_{n + j, n - k}$ for large $n$ as follows,
  \begin{equation} \label{eq:a_n-1_n_eval}
    \begin{split}
      a_{n + j, n - k} = {}& \int^{\infty}_{-\infty} \widetilde p^{(n)}_{n - k}(x) \exp(x) \widetilde q^{(n)}_{n + j}(x) dx \\
      = {}& \frac{1}{h_{n + j}} \int_{[-C_2, C_2]} p^{(n)}_{n - k}(x) \exp(x) q^{(n)}_{n + j}(e^x) e^{-nV(x)} dx + \bigO(e^{-C_1 n}) \\
      = {}& -\frac{1}{h_{n + j}} \oint_{\Gamma} p^{(n)}_{n - k}(z) \exp(z) Cq^{(n)}_{n + j}(z) dz + \bigO(e^{-C_1n}),
    \end{split}
  \end{equation}
  where the Cauchy transform $Cq^{(n)}_{n + j}(z)$ is defined in \eqref{eq:Cauchy_trans}, and $\Gamma$ is a counterclockwise oriented contour surrounding the support of the equilibrium measure $[a, b]$ and intersecting with the real line at the points $-C_2, C_2$. See Figure \ref{fig:Gamma}.
  \begin{figure}[htb]
    \centering
    \includegraphics{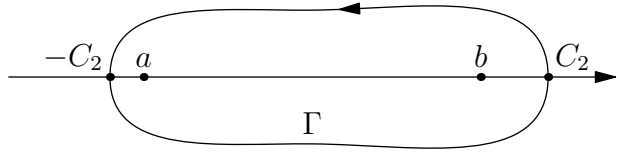}
    \caption{Contour $\Gamma$.}
    \label{fig:Gamma}
  \end{figure}
  
  Below we compute the coefficients $a_{n + j, n - k}$. We need to use the mapping $\Jlike_{c_1(1), c_0(1)}(s)$ defined in \eqref{sab}, and abbreviate it as $\Jlike(s)$ here. We also need $\Jinv_1$, an inverse functions of $\Jlike$, as defined in \cite[(1.26)]{Claeys-Wang11}. Recall \eqref{eq:asy_of_p_n+k}--\eqref{as hnk} describing the outer asymptotics in Lemma \ref{lem:asy_collected}\ref{enu:lem:asy:2}. Substituting those in \eqref{eq:a_n-1_n_eval} and writing $s=\Jinv_1(z)$, we obtain as $n\to\infty$,
  \begin{equation}
    a_{n + j, n - k} = \frac{c^{-j - k - 1}_1}{2\pi i} \oint_{\Gamma} \frac{(s + \frac{1}{2})(s - \frac{1}{2})^{-j - k}}{s^2 - \frac{1}{4} - \frac{1}{c_1}} e^z (1 + \bigO(n^{-1})) dz + \bigO(e^{-C_1n}). \end{equation}    
 Using $\Jlike$, we get
  \begin{equation}
    \begin{split}    
      a_{n + j, n - k} = {}& \frac{c^{-j - k - 1}_1}{2\pi i} \oint_{\Gamma'} \frac{(s + \frac{1}{2})(s - \frac{1}{2})^{-j - k}}{s^2 - \frac{1}{4} - \frac{1}{c_1}} e^{\Jlike(s)} \frac{d\Jlike(s)}{ds} (1 + \bigO(n^{-1})) ds + \bigO(e^{-C_1n}) \\
      = {}& \frac{c^{-j - k}_1 e^{c_0}}{2\pi i} \oint_{\Gamma'} (s + \frac{1}{2})(s - \frac{1}{2})^{-j - k - 2} e^{c_1 s} (1 + \bigO(n^{-1})) ds + \bigO(e^{-C_1n}),
    \end{split}
  \end{equation}
  with $\Gamma'=\Jinv_1(\Gamma)$, and \eqref{estimate a} follows from a simple residue calculus.
\end{proof}

\begin{proof}[Proof of Corollary \ref{lemma sine4}]
  We need to estimate
  \begin{equation}\label{eq lemma4 1}
    F_{n,M}:= e^{n(F(u) - F(v))} \left( \sum^{n - 1}_{j = n - M} \sum^{j + M}_{k = n} a_{k, j} \widetilde{p}^{(n)}_k(u) \widetilde{q}^{(n)}_j(v) - a_{n - 1, n} \widetilde{p}^{(n)}_{n - 1}(u) \widetilde{q}^{(n)}_n(v) \right). 
  \end{equation}
  Recall that, for $u,v$ near the point $x^*$ which lies in $(a, b)$, the interior of the support of the equilibrium measure $\mu = \mu_1$, the asymptotics of $e^{nF(u)} \widetilde{p}^{(n)}_{n + k}(u)$ and $e^{-nF(v)} \widetilde{q}^{(n)}_{n + k}(e^v)$ are given by \eqref{eq:bulk_p} and \eqref{eq:bulk_q}, with $j = n$ and $t = j/n = 1$ (here $F = F_1$). Then
  \begin{equation}
    e^{nF(u)} \widetilde{p}^{(n)}_{n + k}(u) \widetilde{q}^{(n)}_{n + j}(v) e^{-nF(v)} = A_{k, j}(u, v) + \bigO(n^{-1}),
  \end{equation}
  with ($\mu = \mu_1$, $r_k = r_{1, k}$, $\hat{r}_k = \hat{r}_{1, k}$, $\theta_k = \theta_{1, k}$ and $\hat{\theta}_k = \hat{\theta}_{1, k}$ in \eqref{eq:bulk_p} and \eqref{eq:bulk_q})
  \begin{multline}
    A_{k, j}(u, v) = {} \frac{1}{2\pi} c^{-j - \frac{1}{2}}_1 e^{-j(\frac{c_1}{2} + c_0)} r_k(x^*) \left[ \cos \left( n\pi \int^b_{x^*} d\mu(t) + \theta_k(x^*) - \xi \right) \right] \\
    \times \hat{r}_j(x^*) \left[ \cos \left( n\pi \int^b_{x^*} d\mu(t) + \hat{\theta}_j(x^*) - \eta \right) \right].
  \end{multline}
  Now we denote $G_k = G_{1, k}$ and $\hat{G}_k = \hat{G}_{t, k}$ where $G_{t, k}$ and $\hat{G}_{t, k}$ are defined in \eqref{def Gk}. Using the relations $\Jinv_-(\cdot) = \overline{\Jinv_+(\cdot)}$, $G_k(\Jinv_-(\cdot)) = \overline{G_k(\Jinv_+(\cdot))}$, this can be further computed as
  \begin{equation} \label{eq:approx_of_each_term}
    \begin{split}
      A_{k, j}(u, v) = {}&  \frac{1}{2\pi} c^{-j - \frac{1}{2}}_1 e^{-j(\frac{c_1}{2} + c_0)} 4 \Re \left( G_k(\Jinv_+(x^*)) e^{in\pi \int^b_{x^*} d\mu - i\xi} \right) \Re \left( \hat{G}_j(\Jinv_-(x^*)) e^{in\pi \int^b_{x^*} d\mu - i\eta} \right) \\
      = {}& \frac{1}{2\pi} c^{-j - \frac{1}{2}}_1 e^{-j(\frac{c_1}{2} + c_0)} \left( G_k(\Jinv_+(x^*)) e^{in\pi \int^b_{x^*} d\mu - i\xi} + G_k(\Jinv_-(x^*)) e^{-in\pi \int^b_{x*} d\mu + i\xi} \right) \\
      & \times \left( \hat{G}_j(\Jinv_+(x^*)) e^{-in\pi \int^b_{x^*} d\mu + i\eta} + \hat{G}_j(\Jinv_-(x^*)) e^{in\pi \int^b_{x^*} d\mu - i\eta} \right) \\
      = {}& \frac{1}{2\pi} c^{-j - \frac{1}{2}}_1 e^{-j(\frac{c_1}{2} + c_0)} \\
      &\times\ \left[ G_k(\Jinv_+(x^*))\hat{G}_j(\Jinv_-(x^*)) e^{2in\pi \int^b_{x^*} d\mu - i(\xi + \eta)} \right. 
       + G_k(\Jinv_-(x^*))\hat{G}_j(\Jinv_+(x^*)) e^{-2in\pi \int^b_{x^*} d\mu + i(\xi + \eta)} \\
      &\qquad   + G_k(\Jinv_+(x^*))\hat{G}_j(\Jinv_+(x^*)) e^{-i(\xi - \eta)} + \left.  G_k(\Jinv_-(x^*))\hat{G}_j(\Jinv_-(x^*)) e^{i(\xi - \eta)} \right] \\
      = {}& \frac{1}{\pi} c^{-j - \frac{1}{2}}_1 e^{-j(\frac{c_1}{2} + c_0)} \left[ \Re (G_k(\Jinv_+(x^*)) \hat{G}_j(\Jinv_-(x^*)) e^{2in\pi \int^b_{x^*} d\mu - i(\xi + \eta)}) \right. \\
      & \phantom{\frac{1}{2\pi} c^{-j - \frac{1}{2}}_1 e^{-j(\frac{c_1}{2} + c_0)}} + \left. \vphantom{e^{in\pi \left( \int^b_x d\mu + \int^b_y d\mu \right)}} \Re (G_k(\Jinv_+(x^*)) \hat{G}_j(\Jinv_+(x^*))) e^{-i(\xi - \eta)} \right].
    \end{split}
  \end{equation}
  We want to find asymptotics for $F_{n,M}$ in \eqref{eq lemma4 1} by substituting the formula for $A_{k,-l}$. We have for any fixed $M$, as $n \to \infty$,
  \begin{equation}
    \begin{split}
      F_{n,M} = {}& \left(\sum^M_{l = 1} \sum^{M - l}_{k = 0} a_{n + k, n - l} A_{k, -l}(u, v) - a_{n - 1, n} A_{-1, 0}(u, v)\right) + \bigO(n^{-1}) \\
      = {}& \left( \sum^M_{l = 1} \sum^{M - l}_{k = 0} \alpha_{k + l} \frac{1}{\pi} c^{l - \frac{1}{2}}_1 e^{l(\frac{c_1}{2} + c_0)} \left[ \Re (G_k(\Jinv_+(x^*)) \hat{G}_{-l}(\Jinv_-(x^*)) e^{2in\pi \int^b_{x^*} d\mu - i(\xi + \eta)}) \right. \right. \\
      & \phantom{\smash{\left( \sum^M_{l = 1} \sum^{M - 1}_{k = 0} \alpha_{k + l} \frac{1}{\pi} c^{l - \frac{1}{2}}_1 e^{l(\frac{c_1}{2} + c_0)} \right.}} - \left.  \Re (G_k(\Jinv_+(x^*)) \hat{G}_{-l}(\Jinv_+(x^*)) e^{-i(\xi - \eta)}) \right]\\
      &  - \alpha_{-1} \frac{1}{\pi} c^{-\frac{1}{2}}_1 \left[ \Re (G_{-1}(\Jinv_+(x^*)) \hat{G}_0(\Jinv_-(x^*)) e^{2in\pi \int^b_{x^*} d\mu - i(\xi + \eta)}) \right.  \\
      & \left. \vphantom{\sum^M_{l = 1} \sum^{M - l}_{k = 0}} \phantom{\smash{- \alpha_{-1} \frac{1}{\pi} c^{-\frac{1}{2}}_1}} \left. - \Re (G_{-1}(\Jinv_+(x^*)) \hat{G}_0(\Jinv_+(x^*)) e^{-i(\xi - \eta)}) \right] \right) + \bigO(n^{-1}).
    \end{split}
  \end{equation}
  Next, we define
  \begin{multline} \label{eq:first_tricky_id}
    I_5:=-c_1 G_{-1}(\Jinv_+(x^*))\hat{G}_0(\Jinv_+(x^*)) \\+ \sum^{\infty}_{l = 1} \sum^{\infty}_{k = 0} \left( \frac{c_1}{(1 + k + l)!} + \frac{1}{(k + l)!} \right) c^l_1 e^{l(\frac{c_1}{2} + c_0)} G_k(\Jinv_+(x^*)) \hat{G}_{-l}(\Jinv_+(x^*)),
  \end{multline}
  \begin{multline} \label{eq:second_tricky_id}
    I_6:=-c_1 G_{-1}(\Jinv_+(x^*))\hat{G}_0(\Jinv_-(x^*)) \\+ \sum^{\infty}_{l = 1} \sum^{\infty}_{k = 0} \left( \frac{c_1}{(1 + k + l)!} + \frac{1}{(k + l)!} \right) c^l_1 e^{l(\frac{c_1}{2} + c_0)} G_k(\Jinv_+(x^*)) \hat{G}_{-l}(\Jinv_-(x^*)),
  \end{multline}
  and by \eqref{estimate a}, we have the identity between absolute convergent series
  \begin{equation}
    \begin{split}
      & \frac{1}{\pi} c^{-\frac{1}{2}}_1 e^{\frac{c_1}{2} + c_0} \left( \Re[I_6 e^{2in\pi \int^b_{x^*} d\mu - i(\xi + \eta)}] - \Re[I_5 e^{-i(\xi - \eta)}] \right) \\
      = {}& \sum^{\infty}_{l = 1} \sum^{\infty}_{k = 0} \alpha_{k + l} \frac{1}{\pi} c^{l - \frac{1}{2}}_1 e^{l(\frac{c_1}{2} + c_0)} \left[ \Re (G_k(\Jinv_+(x^*)) \hat{G}_{-l}(\Jinv_-(x^*)) e^{2in\pi \int^b_{x^*} d\mu - i(\xi + \eta)}) \right. \\
      & \phantom{\smash{ \sum^{\infty}_{l = 1} \sum^{\infty}_{k = 0} \alpha_{k + l} \frac{1}{\pi} c^{l - \frac{1}{2}}_1 e^{l(\frac{c_1}{2} + c_0)}}} - \left.  \Re (G_k(\Jinv_+(x^*)) \hat{G}_{-l}(\Jinv_+(x^*)) e^{-i(\xi - \eta)}) \right]\\
      &  - \alpha_{-1} \frac{1}{\pi} c^{-\frac{1}{2}}_1 \left[ \Re (G_{-1}(\Jinv_+(x^*)) \hat{G}_0(\Jinv_-(x^*)) e^{2in\pi \int^b_{x^*} d\mu - i(\xi + \eta)}) \right.  \\
      & \vphantom{\sum^{\infty}_{l = 1} \sum^{\infty}_{k = 0}} \phantom{\smash{- \alpha_{-1} \frac{1}{\pi} c^{-\frac{1}{2}}_1}} \left. - \Re (G_{-1}(\Jinv_+(x^*)) \hat{G}_0(\Jinv_+(x^*)) e^{-i(\xi - \eta)}) \right].
    \end{split}
  \end{equation}
  If we can prove
  \begin{equation} \label{eq:estimate_in_first_y_close_to_x}
    I_5=i\sqrt{c_1} e^{-(\frac{c_1}{2} + c_0)} e^{x^*}, \quad I_6=0,
  \end{equation}
  then the corollary is proved.
  
  To finish the proof, we first consider $I_5$ in \eqref{eq:first_tricky_id}. By the definition \eqref{def Gk} of $G_k$ and $\hat{G}_{-l}$, we get
  \begin{multline} \label{eq:first_id_transf}
    I_5= \frac{i}{\sqrt{c_1}} \frac{\Jinv_+(x^*) + \frac{1}{2}}{\Jinv_+(x^*)^2 - \frac{1}{4} - \frac{1}{c_1}} \\
    \times \left[ -\frac{1}{\Jinv_+(x^*) - \frac{1}{2}} + \sum^{\infty}_{l = 1} \sum^{\infty}_{k = 0} \left( \frac{c_1}{(1 + k + l)!} + \frac{1}{(k + l)!} \right) c^{k + l}_1 (\Jinv_+(x^*) - \frac{1}{2})^{k + l} \right].
  \end{multline}
  By the identities
  \begin{equation}
    \sum^{\infty}_{l = 1} \sum^{\infty}_{k = 0} \frac{1}{(k + l)!} x^{k + l} = xe^x, \quad \sum^{\infty}_{l = 1} \sum^{\infty}_{k = 0} \frac{1}{(1 + k + l)!} x^{k + l} = (1 - \frac{1}{x})e^x + \frac{1}{x},
  \end{equation}
  we have
  \begin{multline} \label{eq:long_sum_evaluated}
    \sum^{\infty}_{l = 1} \sum^{\infty}_{k = 0} \left( \frac{c_1}{(1 + k + l)!} + \frac{1}{(k + l)!} \right) c^{k + l}_1 (\Jinv_+(x^*) - \frac{1}{2})^{k + l} = \\
    \left( c_1(\Jinv_+(x^*) + \frac{1}{2}) - \frac{1}{\Jinv_+(x^*) - \frac{1}{2}} \right) e^{c_1(\Jinv_+(x^*) - \frac{1}{2})} + \frac{1}{\Jinv_+(x^*) - \frac{1}{2}}.
  \end{multline}
  Using the identity $\Jlike(\Jinv_+(x)) = x$, we have
  \begin{equation} \label{eq:J(I(x))}
    (\Jinv_+(x^*) + \frac{1}{2}) e^{c_1(\Jinv_+(x^*) - \frac{1}{2})} = e^{-(\frac{c_1}{2} + c_0)} e^{x^*} (\Jinv_+(x^*) - \frac{1}{2}), \quad \frac{e^{c_1(\Jinv_+(x^*) - \frac{1}{2})}}{\Jinv_+(x^*) - \frac{1}{2}}  = e^{-(\frac{c_1}{2} + c_0)} \frac{e^{x^*}}{\Jinv_+(x^*) + \frac{1}{2}}.
  \end{equation}
  Substituting \eqref{eq:J(I(x))} into \eqref{eq:long_sum_evaluated}, and then substituting the simplified form of \eqref{eq:long_sum_evaluated} into \eqref{eq:first_id_transf}, we simplify \eqref{eq:first_id_transf}, or equivalently \eqref{eq:first_tricky_id}, and obtain the expression for $I_5$ from \eqref{eq:estimate_in_first_y_close_to_x}.
  
  Similarly we simplify \eqref{eq:second_tricky_id} as
  \begin{multline} \label{eq:key_of_second_id}
    \frac{i}{\sqrt{c_1}} \frac{\Jinv_+(x^*) + \frac{1}{2}}{\sqrt{\Jinv_+(x^*)^2 - \frac{1}{4} - \frac{1}{c_1}} \sqrt{\Jinv_-(x^*)^2 - \frac{1}{4} - \frac{1}{c_1}}} \\
    \times \left[ -\frac{1}{\Jinv_+(x^*) - \frac{1}{2}} + \sum^{\infty}_{l = 1} \sum^{\infty}_{k = 0} \left( \frac{c_1}{(1 + k + l)!} + \frac{1}{(k + l)!} \right) c^{k + l}_1 (\Jinv_+(x^*) - \frac{1}{2})^k (\Jinv_-(x^*) - \frac{1}{2})^l \right].
  \end{multline}
  Then, using the identities
  \begin{equation}
    \sum^{\infty}_{l = 1} \sum^{\infty}_{k = 0} \frac{1}{(k + l)!} x^ky^l = \frac{y}{y - x}(e^y - e^x), \quad \sum^{\infty}_{l = 1} \sum^{\infty}_{k = 0} \frac{1}{(1 + k + l)!} x^ky^l = \frac{y}{y - x}(\frac{e^y}{y} - \frac{e^x}{x}) + \frac{1}{x},
  \end{equation}
  we have analogous to \eqref{eq:long_sum_evaluated},
  \begin{multline} \label{eq:formula_with_cancellation_in_parentheses}
    \sum^{\infty}_{l = 1} \sum^{\infty}_{k = 0} \left( \frac{c_1}{(1 + k + l)!} + \frac{1}{(k + l)!} \right) c^{k + l}_1 (\Jinv_+(x^*) - \frac{1}{2})^k (\Jinv_-(x^*) - \frac{1}{2})^l = \\
    \frac{\Jinv_-(x^*) - \frac{1}{2}}{\Jinv_-(x^*) - \Jinv_+(x^*)} \left( \frac{\Jinv_-(x^*) + \frac{1}{2}}{\Jinv_-(x^*) - \frac{1}{2}} e^{c_1(\Jinv_-(x^*) - \frac{1}{2})} - \frac{\Jinv_+(x^*) + \frac{1}{2}}{\Jinv_+(x^*) - \frac{1}{2}} e^{c_1(\Jinv_+(x^*) - \frac{1}{2})} \right) + \frac{1}{\Jinv_+(x) - \frac{1}{2}}.
  \end{multline}
  Using the relation that $\Jlike(\Jinv_+(x^*)) = \Jlike(\Jinv_-(x^*))$, we find that the two terms between the parentheses in \eqref{eq:formula_with_cancellation_in_parentheses} cancel each other. Substituting \eqref{eq:formula_with_cancellation_in_parentheses} into \eqref{eq:key_of_second_id}, we find that $I_6 = 0$. Thus the proof is complete.
\end{proof}

\section*{Acknowledgements}
The authors were partially supported by the Fonds de la Recherche Scientifique-FNRS under EOS project O013018F (T.~C.), National Natural Science Foundation of China under grant number 11871425 (D.~W.), and the University of Chinese Academy of Sciences start-up grant 118900M043 (D.~W.).

% \bibliography{../bibliography/bibliography.bib}

\begin{thebibliography}{10}

\bibitem{Aptekarev-Bleher-Kuijlaars05}
A.~I. Aptekarev, P.~M. Bleher, and A.~B.~J. Kuijlaars.
\newblock Large {$n$} limit of {G}aussian random matrices with external source.
  {II}.
\newblock {\em Comm. Math. Phys.}, 259(2):367--389, 2005.

\bibitem{Aptekarev-Lysov-Tulyakov11}
A.~I. Aptekarev, V.~G. Lysov, and D.~N. Tulyakov.
\newblock Random matrices with an external source and the asymptotics of
  multiple orthogonal polynomials.
\newblock {\em Mat. Sb.}, 202(2):3--56, 2011.

\bibitem{Baik-Wang10a}
J.~Baik and D.~Wang.
\newblock On the largest eigenvalue of a {H}ermitian random matrix model with
  spiked external source {I}. {R}ank 1 case.
\newblock {\em Int. Math. Res. Not. IMRN}, (22):5164--5240, 2011.

\bibitem{Baik-Wang10}
J.~Baik and D.~Wang.
\newblock On the largest eigenvalue of a {H}ermitian random matrix model with
  spiked external source {II}: {H}igher rank cases.
\newblock {\em Int. Math. Res. Not. IMRN}, (14):3304--3370, 2013.

\bibitem{Beenakker97}
C.~W.~J. Beenakker.
\newblock Random-matrix theory of quantum transport.
\newblock {\em Rev. Mod. Phys.}, 69:731--808, Jul 1997.

\bibitem{Bertola-Buckingham-Lee-Pierce11a}
M.~Bertola, R.~Buckingham, S.~Y. Lee, and V.~Pierce.
\newblock Spectra of random {H}ermitian matrices with a small-rank external
  source: the critical and near-critical regimes.
\newblock {\em J. Stat. Phys.}, 146(3):475--518, 2012.

\bibitem{Bertola-Buckingham-Lee-Pierce11}
M.~Bertola, R.~Buckingham, S.~Y. Lee, and V.~Pierce.
\newblock Spectra of random {H}ermitian matrices with a small-rank external
  source: the supercritical and subcritical regimes.
\newblock {\em J. Stat. Phys.}, 153(4):654--697, 2013.

\bibitem{Betea-Occelli20a}
D.~Betea and A.~Occelli.
\newblock Discrete and continuous {M}uttalib--{B}orodin processes {I}: the hard
  edge, 2020.
\newblock arXiv:2010.15529.

\bibitem{Betea-Occelli20}
D.~Betea and A.~Occelli.
\newblock Muttalib-{B}orodin plane partitions and the hard edge of random
  matrix ensembles.
\newblock {\em S\'{e}m. Lothar. Combin.}, 85B:Art. 8, 12, 2021.

\bibitem{Bleher-Delvaux-Kuijlaars10}
P.~Bleher, S.~Delvaux, and A.~B.~J. Kuijlaars.
\newblock Random matrices model with external source and a constrained vector
  equilibrium problem.
\newblock {\em Comm. Pure Appl. Math.}, 64(1):116--160, 2011.

\bibitem{Bleher-Kuijlaars04}
P.~Bleher and A.~B.~J. Kuijlaars.
\newblock Large {$n$} limit of {G}aussian random matrices with external source.
  {I}.
\newblock {\em Comm. Math. Phys.}, 252(1-3):43--76, 2004.

\bibitem{Bleher-Kuijlaars05}
P.~M. Bleher and A.~B.~J. Kuijlaars.
\newblock Integral representations for multiple {H}ermite and multiple
  {L}aguerre polynomials.
\newblock {\em Ann. Inst. Fourier (Grenoble)}, 55(6):2001--2014, 2005.

\bibitem{Bleher-Kuijlaars07}
P.~M. Bleher and A.~B.~J. Kuijlaars.
\newblock Large {$n$} limit of {G}aussian random matrices with external source.
  {III}. {D}ouble scaling limit.
\newblock {\em Comm. Math. Phys.}, 270(2):481--517, 2007.

\bibitem{Bloom-Levenberg-Totik-Wielonsky17}
T.~Bloom, N.~Levenberg, V.~Totik, and F.~Wielonsky.
\newblock Modified logarithmic potential theory and applications.
\newblock {\em Int. Math. Res. Not. IMRN}, (4):1116--1154, 2017.

\bibitem{Borodin99}
A.~Borodin.
\newblock Biorthogonal ensembles.
\newblock {\em Nuclear Phys. B}, 536(3):704--732, 1999.

\bibitem{Brezin-Hikami98}
E.~Br{\'e}zin and S.~Hikami.
\newblock Level spacing of random matrices in an external source.
\newblock {\em Phys. Rev. E (3)}, 58(6, part A):7176--7185, 1998.

\bibitem{Butez17}
R.~Butez.
\newblock Large deviations for biorthogonal ensembles and variational
  formulation for the {D}ykema-{H}aagerup distribution.
\newblock {\em Electron. Commun. Probab.}, 22:Paper No. 37, 11, 2017.

\bibitem{Charlier21}
C.~Charlier.
\newblock Asymptotics of {M}uttalib-{B}orodin determinants with
  {F}isher-{H}artwig singularities.
\newblock {\em Selecta Math. (N.S.)}, 28(3):Paper No. 50, 60, 2022.

\bibitem{Charlier-Lenells-Mauersberger19}
C.~Charlier, J.~Lenells, and J.~Mauersberger.
\newblock Higher order large gap asymptotics at the hard edge for
  {M}uttalib-{B}orodin ensembles.
\newblock {\em Comm. Math. Phys.}, 384(2):829--907, 2021.

\bibitem{Claeys-Girotti-Stivigny19}
T.~Claeys, M.~Girotti, and D.~Stivigny.
\newblock Large gap asymptotics at the hard edge for product random matrices
  and {M}uttalib-{B}orodin ensembles.
\newblock {\em Int. Math. Res. Not. IMRN}, (9):2800--2847, 2019.

\bibitem{Claeys-Wang11}
T.~Claeys and D.~Wang.
\newblock Random matrices with equispaced external source.
\newblock {\em Comm. Math. Phys.}, 328(3):1023--1077, 2014.

\bibitem{Credner-Eichelsbacher15}
K.~Credner and P.~Eichelsbacher.
\newblock Large deviations for the largest eigenvalue of disordered bosons and
  disordered fermionic systems, 2015.
\newblock arXiv:1503.00984.

\bibitem{Desrosiers-Forrester08}
P.~Desrosiers and P.~J. Forrester.
\newblock A note on biorthogonal ensembles.
\newblock {\em J. Approx. Theory}, 152(2):167--187, 2008.

\bibitem{Eichelsbacher-Sommerauer-Stolz11}
P.~Eichelsbacher, J.~Sommerauer, and M.~Stolz.
\newblock Large deviations for disordered bosons and multiple orthogonal
  polynomial ensembles.
\newblock {\em J. Math. Phys.}, 52(7):073510, 16, 2011.

\bibitem{Eichinger-Lukic-Simanek21}
B.~Eichinger, M.~Luki{\'c}, and B.~Simanek.
\newblock An approach to universality using {W}eyl $m$-functions, 2021.
\newblock arXiv:2108.01629.

\bibitem{Erdos-Ramirez-Schlein-Tao-Vu-Yau10}
L.~Erd\H{o}s, J.~Ram\'{\i}rez, B.~Schlein, T.~Tao, V.~Vu, and H.-T. Yau.
\newblock Bulk universality for {W}igner {H}ermitian matrices with
  subexponential decay.
\newblock {\em Math. Res. Lett.}, 17(4):667--674, 2010.

\bibitem{Erdos-Ramirez-Schlein-Yau10}
L.~Erd\H{o}s, J.~A. Ram\'{\i}rez, B.~Schlein, and H.-T. Yau.
\newblock Universality of sine-kernel for {W}igner matrices with a small
  {G}aussian perturbation.
\newblock {\em Electron. J. Probab.}, 15:no. 18, 526--603, 2010.

\bibitem{Erdos-Yau17}
L.~Erd\H{o}s and H.-T. Yau.
\newblock {\em A dynamical approach to random matrix theory}, volume~28 of {\em
  Courant Lecture Notes in Mathematics}.
\newblock Courant Institute of Mathematical Sciences, New York; American
  Mathematical Society, Providence, RI, 2017.

\bibitem{Eynard-Orantin09}
B.~Eynard and N.~Orantin.
\newblock Topological recursion in enumerative geometry and random matrices.
\newblock {\em J. Phys. A}, 42(29):293001, 117, 2009.

\bibitem{Forrester-Wang15}
P.~J. Forrester and D.~Wang.
\newblock Muttalib-{B}orodin ensembles in random matrix theory---realisations
  and correlation functions.
\newblock {\em Electron. J. Probab.}, 22:Paper No. 54, 43, 2017.

\bibitem{Gautie-Le_Doussal-Majumdar-Schehr19}
T.~Gauti\'{e}, P.~Le~Doussal, S.~N. Majumdar, and G.~Schehr.
\newblock Non-crossing {B}rownian paths and {D}yson {B}rownian motion under a
  moving boundary.
\newblock {\em J. Stat. Phys.}, 177(5):752--805, 2019.

\bibitem{Grela-Majumdar-Schehr21}
J.~Grela, S.~N. Majumdar, and G.~Schehr.
\newblock Non-intersecting {B}rownian bridges in the flat-to-flat geometry.
\newblock {\em J. Stat. Phys.}, 183(3):Paper No. 49, 35, 2021.

\bibitem{Harish-Chandra57}
Harish-Chandra.
\newblock Differential operators on a semisimple {L}ie algebra.
\newblock {\em Amer. J. Math.}, 79:87--120, 1957.

\bibitem{Itzykson-Zuber80}
C.~Itzykson and J.~B. Zuber.
\newblock The planar approximation. {II}.
\newblock {\em J. Math. Phys.}, 21(3):411--421, 1980.

\bibitem{Johansson01a}
K.~Johansson.
\newblock Universality of the local spacing distribution in certain ensembles
  of {H}ermitian {W}igner matrices.
\newblock {\em Comm. Math. Phys.}, 215(3):683--705, 2001.

\bibitem{Johansson12}
K.~Johansson.
\newblock Universality for certain {H}ermitian {W}igner matrices under weak
  moment conditions.
\newblock {\em Ann. Inst. Henri Poincar\'{e} Probab. Stat.}, 48(1):47--79,
  2012.

\bibitem{Kuijlaars-Molag19}
A.~B.~J. Kuijlaars and L.~D. Molag.
\newblock The local universality of {M}uttalib-{B}orodin biorthogonal ensembles
  with parameter {$\theta=\frac 12$}.
\newblock {\em Nonlinearity}, 32(8):3023--3081, 2019.

\bibitem{Lubinsky09}
D.~S. Lubinsky.
\newblock A new approach to universality limits involving orthogonal
  polynomials.
\newblock {\em Ann. of Math. (2)}, 170(2):915--939, 2009.

\bibitem{Lueck-Sommers-Zirnbauer06}
T.~Lueck, H.-J. Sommers, and M.~R. Zirnbauer.
\newblock Energy correlations for a random matrix model of disordered bosons.
\newblock {\em J. Math. Phys.}, 47(10):103304, 24, 2006.

\bibitem{Molag20}
L.~D. Molag.
\newblock The local universality of {M}uttalib-{B}orodin ensembles when the
  parameter {$\theta$} is the reciprocal of an integer.
\newblock {\em Nonlinearity}, 34(5):3485--3564, 2021.

\bibitem{Muttalib95}
K.~A. Muttalib.
\newblock Random matrix models with additional interactions.
\newblock {\em J. Phys. A}, 28(5):L159--L164, 1995.

\bibitem{Nikishin-Sorokin91}
E.~M. Nikishin and V.~N. Sorokin.
\newblock {\em Rational approximations and orthogonality}, volume~92 of {\em
  Translations of Mathematical Monographs}.
\newblock American Mathematical Society, Providence, RI, 1991.
\newblock Translated from the Russian by Ralph P. Boas.

\bibitem{Swiderski-Trojan22}
G.~{\'S}widerski and B.~Trojan.
\newblock Orthogonal polynomials with periodically modulated recurrence
  coefficients in the {J}ordan block case, 2020.
\newblock arXiv:2008.07296.

\bibitem{Swiderski-Trojan21a}
G.~\'{S}widerski and B.~Trojan.
\newblock Asymptotic behavior of {C}hristoffel-{D}arboux kernel via three-term
  recurrence relation {II}.
\newblock {\em J. Approx. Theory}, 261:Paper No. 105496, 48, 2021.

\bibitem{Swiderski-Trojan21}
G.~\'{S}widerski and B.~Trojan.
\newblock Asymptotic behaviour of {C}hristoffel-{D}arboux kernel via three-term
  recurrence relation {I}.
\newblock {\em Constr. Approx.}, 54(1):49--116, 2021.

\bibitem{Swiderski-Trojan22a}
G.~{\'S}widerski and B.~Trojan.
\newblock Orthogonal polynomials with periodically modulated recurrence
  coefficients in the {J}ordan block case {II}, 2021.
\newblock arXiv:2107.11154.

\bibitem{Tracy-Widom98}
C.~A. Tracy and H.~Widom.
\newblock Correlation functions, cluster functions, and spacing distributions
  for random matrices.
\newblock {\em J. Statist. Phys.}, 92(5-6):809--835, 1998.

\bibitem{Wang-Zhang21}
D.~Wang and L.~Zhang.
\newblock A vector {R}iemann-{H}ilbert approach to the {M}uttalib-{B}orodin
  ensembles.
\newblock {\em J. Funct. Anal.}, 282(7):Paper No. 109380, 84, 2022.

\bibitem{Zinn_Justin97}
P.~Zinn-Justin.
\newblock Random {H}ermitian matrices in an external field.
\newblock {\em Nuclear Phys. B}, 497(3):725--732, 1997.

\end{thebibliography}
% \bibliographystyle{abbrv}

\end{document}